\documentclass[11pt, twoside, leqno]{article}

\usepackage{amssymb,amsmath}
\usepackage{amsthm}
\usepackage{graphicx}

\numberwithin{equation}{section}

\newtheorem{thm}{Theorem}[section]

\newtheorem{lemma}[thm]{Lemma}

\newtheorem*{thm1}{Theorem 1}
\newtheorem*{thm2}{Theorem 2}
\newtheorem*{thm3}{Theorem 3}
\newtheorem*{thm4}{Theorem 4}
\newtheorem*{thm5}{Theorem 5}
\newtheorem*{thm6}{Theorem 6}

\newtheorem*{corollary}{Corollary}
\newtheorem*{corollary1}{Corollary 1}
\newtheorem*{corollary2}{Corollary 2}

\theoremstyle{definition}
\newtheorem*{example}{Example}
\newtheorem*{rmk}{Remark}
\newtheorem*{rmk1}{Remark 1}
\newtheorem*{rmk2}{Remark 2}

\newtheorem*{defn}{Definition}

\long\def\symbolfootnote[#1]#2{\begingroup%
\def\thefootnote{\fnsymbol{footnote}}\footnote[#1]{#2}\endgroup}

\numberwithin{equation}{section}

\begin{document}
\newcommand{\ul}{\underline}
\newcommand{\be}{\begin{equation}}
\newcommand{\ee}{\end{equation}}
 \baselineskip = 16pt \noindent

\title{On Fibonacci Polynomial Expressions for Sums of $m$th Powers, their implications for Faulhaber's Formula and some Theorems of Fermat}
\date{October 15, 2015}

\author{M. W. Coffey (Colorado School of Mines) and M. C. Lettington (Cardiff)}
\maketitle

\begin{abstract}
Denote by $\Sigma n^m$ the sum of the $m$-th powers of the first $n$ positive integers $1^m+2^m+\ldots +n^m$. Similarly let $\Sigma^r n^m$ be the $r$-fold sum of the $m$-th powers of the first $n$ positive integers,
defined such that $\Sigma^0 n^{m}=n^m$, and then recursively by $\Sigma^{r+1} n^{m}=\Sigma^{r} 1^{m}+\Sigma^{r} 2^{m}+\ldots + \Sigma^{r} n^{m}$. During the early 17th-century, polynomial expressions for the sums $\Sigma^r n^m$ and their factorisation and polynomial basis representation properties were examined by Johann Faulhaber, who published some remarkable theorems relating to these $r$-fold sums in his \emph{Academia Algebrae} (1631).

In this paper we consider families of polynomials related to the Fibonacci and Lucas polynomials which naturally lend themselves to representing sums and differences of integer powers,
as well as $\Sigma^r n^m$. Using summations over polynomial basis representations for monomials we generalise this sum for any polynomial $q(x)\in\mathbb{Q}[x]$, encountering some interesting coefficient families in the process. Other results include using our polynomial expressions to state some theorems of Fermat, where in particular, we obtain an explicit expression for the quotient in Fermat's Little Theorem.

In the final two sections we examine these sums in the context of the three pairs of Stirling number varieties and binomial decompositions. We also derive an expression for $\Sigma n^m$ in terms of the Stirling numbers and Pochhammer symbols, leading to a seemingly new relation between the Stirling numbers and the Bernoulli numbers.
\end{abstract}

\section{Introduction}
\symbolfootnote[0]{2010 \emph{Mathematics Subject Classification}: 11B83; 11B37; 11B39; 11B68; 01A45.\newline
\emph{Key words and phrases}: Special sequences; Fibonacci polynomials; Recurrences; Stirling numbers.\newline
The authors would like to thank Professor M N Huxley for his helpful and perceptive comments. }
Johann Faulhaber, know in his day as \emph{The Great Arithmetician of Ulm} but now almost forgotten, worked out formulae for the integer power sums $1^m+2^m+\ldots +n^m$ in his \emph{Academia Algebrae} (1631) \cite{faulhaber}. The formulas for the sums of squares and cubes of the numbers $1,2,3,\ldots,n$ are given by
\[
\Sigma n^2= 1^2+2^2+3^2+\ldots +n^2=\frac{1}{3}\left (n^3+\frac{3}{2}n^2+\frac{1}{2}n\right )=\frac{1}{6} n (n+1) (2 n+1),
\]
\[
\Sigma n^3=1^3+2^3+3^3+\ldots +n^3=\frac{1}{4}\left (n^4+2n^3+n^2\right )=\frac{1}{4} n^2 (n+1)^2=\left (\Sigma n\right )^2,
\]
where in the notation of Knuth \cite{knuth}, $\Sigma n^m$, denotes the sum of the $m$-th powers of the first $n$ positive integers.

Faulhaber noticed that the sum of the first $n$ cubes is equal to the square of the sum of the first $n$ positive integers, and in exploring this pattern he found that the sums $\Sigma n^m$ could be expressed as polynomials in $\mathbb{Q}[n(n+1)]$ and $(2n+1)\mathbb{Q}[n(n+1)]$, such that
\be
\Sigma n^{2m-1}=\frac{1}{2m} \sum_{j=0}^{m-1} A_j^{(m)} (n(n+1))^{m-j},
\label{eq:in1}
\ee
\be
\Sigma n^{2m}=\frac{2n+1}{(2m+1)(2m+2)} \sum_{j=0}^{m}(m+1-j) A_j^{(m+1)} (n(n+1))^{m-j},
\label{eq:in2}
\ee
for certain coefficients $A_j^{(m)}$, where $A_m^{(m-1)}=0$. The expression $(\ref{eq:in2})$ for the even power sum $\Sigma n^{2m}$, is attributed to Jacobi \cite{jacobi}, who showed
that for $m>j+1$ the coefficients $A_j^{(m)}$ satisfy the recurrence relation
\be
2(m-j)(2m-2j-1)A_j^{(m)}+(m-j+1)(m-j)A_{j-1}^{(m)}=2m(2m-1)A_j^{(m-1)}.
\label{eq:in3}
\ee
Knuth demonstrates that $(\ref{eq:in3})$ is valid for all integers $m$ and also considers the $k+1$ term recurrence relation
\[
A_0^{(m)}=1,\qquad \sum_{i=0}^k \binom{m-i}{2k+1-2i}A_i^{(m)}=0,\qquad k>0,
\]
deducing that
\[
A_m^{(m)}=B_{2m},\qquad 2A_{m-1}^{(m+1)}=(2m+1)(2m+2)A_m^{(m)},\qquad A_{m-2}^{(m)}=\binom{2m}{2} B_{2m-2}.
\]
Here $B_k$ denotes the $k$th Bernoulli number \cite{lett}, \cite{mcl1}, where $B_0=1$, and thereafter.
\be
B_k=-\frac{1}{k+1}\sum_{j=0}^{k-1}\binom{k+1}{j}B_j,
\label{eq:bin1}
\ee
so that $B_k=0$ for all odd values of $k\geq 3$.

Faulhaber did not discover the Bernoulli numbers and so did not realise that they provide the uniform formula
\be
\Sigma n^m=\frac{1}{m+1}\sum_{j=0}^m (-1)^j \binom{m+1}{j}n^{m+1-j}B_j,
\label{eq:bin2}
\ee
which later appeared in Jacob Bernoulli's posthumous masterpiece (referred to as Faulhaber's formula) the \emph{Ars Conjectandi} (1713).

Jacobi expressed ($\ref{eq:bin2}$) in terms of Bernoulli polynomials such that
\be
\Sigma n^m=\frac{1}{m+1}\left (B_{m+1}(n+1)-B_{m+1}(0)\right ),
\label{eq:in4}
\ee
so that $\Sigma n^m$ is a polynomial in $n$ of degree $m+1$.

He used their derivative property $B_m^\prime (x)=m B_{m-1}(x)$, to show that the derivative of $\Sigma n^m$ with respect to $n$ is given by
$m \Sigma n^{m-1} +B_m$, from which $(\ref{eq:in2})$ readily follows from $(\ref{eq:in1})$.

Methods for calculating the $A_j^{(m)}$ coefficients have been of interest and the explicit formula
\be
A_j^{(m)}=(-1)^{m-j}\sum_i \binom{2m}{m-j-i}\binom{m-j+i}{i}\frac{m-j-i}{m-j+i}B_{m+j+i}, \quad 0\leq j<m,
\label{eq:in5}
\ee
was first obtained by Gessel and Viennot \cite{gessel}, who using a combinatorial argument applied to the determinant form of this recurrence, deduced that
$(-1)^j A_j^{(m)}\geq 0$. The numbers $A_j^{(m)}$ can also be obtained by inverting a lower triangular matrix, as shown by Edwards \cite{edwards}.

The formula $(\ref{eq:in5})$, makes it easy to confirm that $A_{m-1}^{(m)}=0$, so that the expression $(\ref{eq:in1})$ for $\Sigma n^{2m-1}$, is a polynomial in $\Sigma n$, and the expression $(\ref{eq:in2})$ for $\Sigma n^{2m}$, is $\Sigma n^2$ times by a polynomial in $\Sigma n$. This pattern was not lost on Faulhaber, who considered similar patterns in the more complex $r$-fold summations of $m$th powers from $1$ to $n$, defined by
\[
\Sigma^0 n^{m}=n^m,\qquad \Sigma^{r+1} n^{m}=\Sigma^{r} 1^{m}+\Sigma^{r} 2^{m}+\ldots + \Sigma^{r} n^{m},
\]
so that so that $\Sigma^r n^m$ is a polynomial in $n$ of degree $m+r$.

He found that $\Sigma^{r} n^{2m}$ can be written as $\Sigma^{r} n^{2}$ times a polynomial in $(n(n+r))$, and that $\Sigma^{r} n^{2m-1}$ can be written as $\Sigma^{r} n$ times a polynomial in $(n(n+r))$, where
\be
\Sigma^{r} n = \binom{n+r}{r+1},\qquad \Sigma^{r} n^2 = \frac{2n+r}{r+2}\binom{n+r}{r+1}.
\label{eq:in6}
\ee
Hence, leaving the even power factor $\frac{2n+r}{r+2}$ we have that
\be
\Sigma^{r} n^{m}=\binom{n+r}{r+1}G_{m}^{(r)}(n),\quad \text{\rm with}\quad G_{m}^{(r)}(n)=\sum_{k=0}^{m-1} g_{m\, k}^{(r)} n^k,
\label{eq:in61}
\ee
so that $G_{m}^{(r)}(n)$ is a polynomial in $n$ of degree $m-1$.

In this paper we consider families of polynomials related to the Fibonacci and Lucas polynomials which naturally lend themselves to representing sums and differences of integer powers. Using these polynomials we obtain expressions for $\Sigma^{} n^{m}$ and $\Sigma^{2} n^m$, which lead to the relations
(Corollary to Theorem 5)
\be
\frac{1}{2m+1}\sum_{j=1}^{m+1}\binom{2m+1}{2j-1}\left (\Sigma^r n^{2j-1}\right ) B_{2m+2-2j}=\Sigma^{r+1} n^{2m}-\frac{1}{2}\Sigma^r n^{2m},
\label{eq:in7}
\ee
\be
\frac{1}{2m}\sum_{j=1}^{m}\binom{2m}{2j}\left (\Sigma^r n^{2j}\right ) B_{2m-2j}=\Sigma^{r+1} n^{2m-1}-\frac{1}{2}\Sigma^r n^{2m-1},
\label{eq:in8}
\ee
where it follows that the equation $(\ref{eq:in7})$ can be divided through by $\frac{1}{2}\Sigma^{r} n^2$ and the equation $(\ref{eq:in8})$ by $\frac{1}{2}\Sigma^{r} n$. After division these expressions can be written as polynomials in $\mathbb{Q}[n(n+r)]$, as the left hand side of each equality is such a polynomial. For example
\[
\Sigma^{4} n^{5}-\frac{1}{2}\Sigma^3 n^{5}=
\frac{1}{240} n (n+1) (n+2) (n+3) (2 n+3) \left(\frac{5}{126} n^2 (n+3)^2+\frac{10}{63} n (n+3)-\frac{17}{63}\right).
\]

A recurring theme in this paper is the concept of an $r$-fold sum over the first $n$ integers to the $m$th powers, and in Theorem 1 this sum is generalised to any polynomial $q(x)\in\mathbb{Q}[x]$, via summations over polynomial basis representations for monomials.

Other results include using our polynomial expressions to state some theorems of Fermat, where in particular, we obtain an explicit expression for the quotient in Fermat's Little Theorem.

Central to our considerations are the binomial coefficients and the more elaborate coefficients $T_k(n)$ and $U_k(n)$, defined by
 \[
 T_k(n)=\binom{n+k+1}{2k+1}+\binom{n+k}{2k+1}=\frac{2n+1}{2k+1}\binom{n+k}{2k},
 \]
 \[
 U_k(n)=\binom{n+k}{2k}+\binom{n+k-1}{2k}=\frac{n}{k}\binom{n+k-1}{2k-1},
 \]
 so that
 \[
 T_k(n)-T_k(n-1)=U_k(n),\qquad U_k(n)-U_k(n-1)=T_{k-1}(n-1),
 \]
giving $\Sigma  U_k(n) =  T_k(n)$.

These coefficients occur in an integer coefficient representation of the sum $\Sigma n^m$, as do the \emph{integer central factorial numbers}  $\mathcal{T}_k^{(r)}$ \cite{riordan} (p273), and in Section 4 we examine these expressions in the context of the bigger picture of the three pairs of Stirling number varieties and binomial decompositions. We also derive a seemingly new expression for $\Sigma n^m$ in terms of the Stirling numbers and Pochhammer symbols, leading to an interesting relation between the Stirling numbers and the Bernoulli numbers.

\section{Relations and Expressions for Families of Polynomials}

\begin{defn}[of Fibonacci and Lucas polynomials]
The Fibonacci polynomials $F_n(x)$, are defined by the recurrence relation
\be
F_{n+1}(x)=x F_n(x) +F_{n-1}(x), \qquad\text{with}\qquad F_1(x)=1,\,\,\, F_2(x)=x,
\label{eq:i17}
\ee
or by the explicit sum formula
\be
F_n(x)=\sum_{j=0}^{[(n-1)/2]}\binom{n-j-1}{j}x^{n-2j-1}
=x^{n-1} \,
_2F_1\left(\frac{1}{2}-\frac{n}{2},1-\frac{n}{2};1-n;-\frac{4}{x^2}\right),
\label{eq:175}
\ee
where $_2F_1$ denotes the Gauss hypergeometric function.

Similarly, the Lucas polynomials $L_n(x)$ can also be defined either by the recurrence relation in $(\ref{eq:i17})$, but with initial values $L_1(x)=x$,
$L_2(x)=x^2+2$, or by the explicit sum formula.
\be
L_n(x)=\sum _{j=0}^{\left\lfloor n/2\right\rfloor }\frac{n}{n-j}  \binom{n-j}{j}
x^{n-2 j}
=x^n \, _2F_1\left(\frac{1}{2}-\frac{n}{2},-\frac{n}{2};1-n;-\frac{4}{x^2}\right).
\label{eq:176}
\ee
 We define the reciprocal type of polynomials $F^{\rm Inv}_n(x)$ and $L^{\rm Inv}_n(x)$, such that
 \[
 F_n^{\rm Inv}(x)=x^{n-1} F_n\left (\frac{1}{x}\right ), \qquad  L_n^{\rm Inv}(x)=x^{n} L_n\left (\frac{1}{x}\right ),
 \]
 so that $F_n^{\rm Inv}(x)$ satisfies the recurrence relation $F_{n+2}^{\rm Inv}(x)=F_{n+1}^{\rm Inv}(x)+x^2 F_n^{\rm Inv}(x)$, as does $L_n^{\rm Inv}(x)$.
\end{defn}

\begin{rmk}[to the definition]
It was shown in \cite{koshy} that the Fibonacci and Lucas polynomials have the divisibility properties
\be
F_n(x)\vert F_m(x) \Leftrightarrow n\vert m,\qquad F_n\left (U_{p-1}\left (\sqrt{5}/2\right )\right )=F_{np}/F_p,
\label{eq:i18}
\ee
and
\be
L_n(x)\vert L_m(x) \Leftrightarrow m=(2k+1)n,\qquad \text{for some integer $k$}.
\label{eq:i185}
\ee
The Fibonacci sequence $1,1,2,3,5,8,13,21,\ldots$, is then obtained by setting $x=1$, so that $F_n=F_n(1)$, and similarly for the Lucas sequence with $L_n=L_n(1)$. Taking $F_n=F^{\rm Inv}_n(1)$ and $L_n=L^{\rm Inv}_n(1)$ also gives the desired sequences and it follows from the definitions that our polynomials
$F^{\rm Inv}_n(x)$ and $L^{\rm Inv}_n(x)$ also exhibit the divisibility properties described above.

\end{rmk}
For consistency, in the following polynomial definitions, the subscript $n$ denotes the degree of the polynomial, apart from $\mathcal{Q}^{\rm Inv}_n(x)$, which is of
degree $n-1$.
\begin{defn}[of polynomials $P_n(x)$, $Q_n(x)$, $P^{\rm Inv}_n(x)$ and $Q^{\rm Inv}_n(x)$]
For natural number $m$, we define $P_n(x)$, $Q_n(x)$, $P^{\rm Inv}_n(x)$ and $Q^{\rm Inv}_n(x)$ to be the polynomials of degree $n$  given by
\be
P_n(x)=\sum_{k=0}^{n}T_k(n) x^{k},\qquad P_n^{\rm Inv}(x)=\sum_{k=0}^{n}T_k(n)x^{n-k}
\label{eq:s2}
\ee
\be
Q_n(x)=\sum_{k=0}^{n}U_k(n)x^k,\qquad Q_n^{\rm Inv}(x)=\sum_{k=0}^{n}U_k(n)x^{m-k},
\label{eq:s21}
\ee
so that
\[
P_n^{\rm Inv}(x)=x^n P_n\left (\frac{1}{x}\right ),\qquad Q_n^{\rm Inv}(x)=x^n Q_n\left (\frac{1}{x}\right ),
\]
where the identity
\[
\mathop{\rm lim}_{k\rightarrow 0}\,\,\frac{j}{k}\binom{j+k-1}{2k-1}=2,
\]
gives $Q_0(x)=2$ and similarly for $Q_0^{\rm Inv}(x)=1$, so that are $Q_n(x)$ and $Q^{\rm Inv}_n(x)$ are well defined.

For example, when $n=3$ we have
\[
P_3(x)=x^3+7 x^2+14 x+7,\qquad P_3^{\rm Inv}(x)=7 x^3+14 x^2+7 x+1,
\]
\[
Q_3(x)=x^3+6 x^2+9 x+2,\qquad Q_3^{\rm Inv}(x)=2 x^3+9 x^2+6 x+1,
\]
\end{defn}

\begin{rmk}[to the definition]
It was shown in Theorem 1 of \cite{HDF}, that
\be
P_n(x)=\frac{1}{\sqrt{x}}L_{2n+1}\left (\sqrt{x}\right )
 ={1 \over \sqrt{x}}\left (F_{2n+2}(\sqrt{x})+F_{2n}(\sqrt{x})\right),
\label{eq:s6}
\ee
and
\be
Q_n(x)=L_{2n}\left (\sqrt{x}\right )
=F_{2n+1}(\sqrt{x})+F_{2n-1}(\sqrt{x}).
\label{eq:s7}
\ee
It can similarly be shown that
\be
 P^{\rm Inv}_n(x)=L^{\rm Inv}_{2n+1}\left (\sqrt{x}\right ),\qquad Q^{\rm Inv}_n(x)=L^{\rm Inv}_{2n}\left (\sqrt{x}\right ).
\label{eq:s8}
\ee
\end{rmk}

\begin{defn}[of polynomials $\mathcal{P}_n(x)$, $\mathcal{Q}_n(x)$, $\mathcal{P}^{\rm Inv}_n(x)$ and $\mathcal{Q}^{\rm Inv}_n(x)$]
For natural number $n$, we define $\mathcal{P}_n(x)$, $\mathcal{Q}_n(x)$, $\mathcal{P}^{\rm Inv}_n(x)$ and $\mathcal{Q}^{\rm Inv}_n(x)$
such that
\be
\mathcal{P}_n(x)=\sum_{k=0}^{n}\binom{n+k}{2k}x^{k},\qquad \mathcal{P}_n^{\rm Inv}(x)=\sum_{k=0}^{n}\binom{n+k}{2k}x^{n-k}
\label{eq:s20}
\ee
\be
\mathcal{Q}_n(x)=\sum_{k=0}^{n}\binom{n+k-1}{2k-1}x^k,\qquad \mathcal{Q}_n^{\rm Inv}(x)=\sum_{k=0}^{n}\binom{n+k-1}{2k-1}x^{n-k},
\label{eq:s210}
\ee
where $\binom{n-1}{-1}=0$ and $\binom{-1}{-1}=1$, so that
\[
\mathcal{Q}_0(x)=\mathcal{Q}_0^{\rm Inv}(x)=1,\qquad\mathcal{P}_n^{\rm Inv}(x)=x^n \mathcal{P}_n\left (\frac{1}{x}\right ),\qquad \mathcal{Q}_n^{\rm Inv}(x)=x^n \mathcal{Q}_n\left (\frac{1}{x}\right ),
\]
\end{defn}

\begin{rmk}[to the definition]
It was also shown in Theorem 1 of \cite{HDF} that
\be
\mathcal{P}_n(x)=F_{2n+1}\left (\sqrt{x}\right ),\qquad
\mathcal{Q}_n(x)=\sqrt{x}\,F_{2n}\left (\sqrt{x}\right ),
\label{eq:q1}
\ee
and from the definitions, it follows similarly that
\be
 \mathcal{P}^{\rm Inv}_n(x)=F^{\rm Inv}_{2n+1}\left (\sqrt{x}\right ),\qquad \mathcal{Q}^{\rm Inv}_n(x)=F^{\rm Inv}_{2n}\left (\sqrt{x}\right ).
 \label{eq:q101}
\ee
\end{rmk}
If one distinguishes between odd and even indices, then the following results can all be stated in terms of the four types of Fibonacci and Lucas polynomials described previously. However, having established the links to these polynomials, for clarity and consistency we will express our results via the eight types of $P$ and $Q$ polynomials.

It was shown by Knuth in \cite{knuth} that given an arbitrary function $f(n)$, defined on the integers, there exists a unique polynomial expression in $n$ of the form
\[
f(n)=\sum_{k\geq 0} a_k \binom{n+\lfloor k/2\rfloor}{k},
\]
for some coefficients $a_k$, where the $a_k$ are integers if and only if $f(n)$ is always an integer. He also states that there is a unique expansion
\[
f(n)=b_0 T_0(n)+b_1 U_1(n)+b_2 T_1(n)+b_3 U_2(n)+b_4 T_2(n)+\ldots,
\]
in which the $b_k$ are integers if and only if $f(n)$ is always an integer.

In the following Lemma, we take a slightly different perspective on this approach by demonstrating that each of our polynomial families form a basis for the vector space of polynomials over some given field. It follows that every polynomial function has a power series expansion in $x$, and we give the example $f(x)=e^{bx}$.

\begin{lemma}[Polynomial basis lemma] Concerning the monomial $x^n$ we find that
\be
x^n=\sum_{j=0}^n (-1)^j\binom{2n+1}{j} P_{n-j}(x)
=(-1)^n \binom{2 n}{n}+\sum _{j=0}^{n-1} (-1)^j \binom{2 n}{j}  Q_{n-j}(x)
\label{eq:lm21}
\ee
and $x^n=$
\be
\sum_{j=0}^n (-1)^j\binom{2n+1}{j} \frac{2 n+1-2 j}{2n+1} \mathcal{P}_{n-j}(x)=
\sum _{j=0}^{n-1} (-1)^j \binom{2 n}{j} \frac{2 n-2 j}{2n} \mathcal{Q}_{n-j}(x).
\label{eq:lm22}
\ee
Hence for $m=0,1,2,\ldots$ each of the four families of polynomials generated by $P_m(x)$, $Q_m(x)$, $\mathcal{P}_m(x)$ and $\mathcal{Q}_m(x)$ form
a polynomial basis.

\end{lemma}
\begin{corollary}
Suppose that a function $f$ has a power series expansion,
$$f(x)=\sum_{n=0}^\infty a_nx^n,$$
where $a_n=f^{(n)}(0)/n!$.  Then in terms of the polynomials $P_n(x)$ we obtain
$$f(x)=\sum_{j=0}^\infty c_jP_j(x),$$
where the expansion coefficients are given by
$$c_j=(-1)^j \sum_{n=j}^\infty (-1)^n{{2n+1} \choose {n-j}}a_n.$$
Similar series expansions can be obtained in terms of the polynomials $Q_n(x)$, $\mathcal{P}_n(x)$ and $\mathcal{Q}_n(x)$.
\end{corollary}
\begin{proof}
The identities in ($\ref{eq:lm21}$) and ($\ref{eq:lm22}$) follow by induction on $n$. To see the Corollary, by ($\ref{eq:lm21}$) we have
\[
\sum_{j=0}^\infty c_jP_j(x)=\sum_{j=0}^\infty\left ( (-1)^j \sum_{n=j}^\infty (-1)^n{{2n+1} \choose {n-j}}a_n \right ) P_j(x)=\sum_{n=0}^\infty a_nx^n,
\]
as required.
\end{proof}

\begin{example}[of the Corollary]  Let $f(x)=e^{bx}$, with $a_n=b^n/n!$.  Then we obtain the following
representation, wherein $I_j$ denotes the modified Bessel function of the first kind.
$$f(x)=e^{bx}=e^{-2b}\sum_{j=0}^\infty \left( I_j(2b)-I_{j+1}(2b)\right )P_j(x).$$
Hence we have the identity
$$e^b=e^{-2b}\sum_{j=0}^\infty \left (I_j(2b)-I_{j+1}(2b)\right )L_{2j+1}.$$

Similarly, if $f(x)=\sum_{n=0}^\infty a_nx^n$, then
$$f(x)=\sum_{j=0}^\infty c_j^Q Q_j(x)-c_0^Q,$$
where
$$c_j^Q=(-1)^j \sum_{n=j}^\infty (-1)^n {{2n} \choose {n-j}}a_n, ~~c_0^Q=\sum_{n=0}^\infty (-1)^n {{2n} \choose n}a_n.$$
In particular,
$$f(x)=e^{bx}=e^{-2b}\left(\sum_{j=0}^\infty I_j(2b)Q_j(x)-I_0(2b)\right)$$
$$=e^{-2b}\left(\sum_{j=1}^\infty I_j(2b)Q_j(x)+I_0(2b)\right),$$
and it follows that
$$e^b=e^{-2b}\left(\sum_{j=0}^\infty I_j(2b)(F_{2j-1}(x)+F_{2j+1}(x))-I_0(2b)\right)$$
$$=e^{-2b}\left(\sum_{j=1}^\infty I_j(2b)(F_{2j-1}(x)+F_{2j+1}(x))+I_0(2b)\right).$$
\end{example}

\begin{rmk}[regarding the method of finite differences]
If $f(x)$ is a polynomial of degree $m$ with integer coefficients, then the $a_k$ will be integers, and for $k>m$ will be zero. One way of viewing this is
via the method of finite differences. Define $\Delta f(n)=f(n+1)-f(n)$, so that $\Delta^0 f(n)=f(n)$, and in general the $r$th finite difference is given by the sum
\[
\Delta^r f(n)=\sum_{j=0}^r (-1)^j \binom{r}{j}f(n+r-j).
\]
Then the above finite difference sum holds for all $r\geq 0$, and for $r=m+t>m$ say, with $t\geq 1$, the sum degenerates to
\[
\sum_{j=0}^{m+t} (-1)^j \binom{m+t}{j}=0,
\]
yielding
\be
f(x)=-\sum_{j=1}^{m+t} (-1)^j \binom{m+t}{j}f(x-j).
\label{eq:fds}
\ee
For example, taking $f(x)=x^m$ and integer $t\geq 1$ in $(\ref{eq:fds})$ gives
\[
n^{m}=\sum_{j=1}^{m+t} (-1)^{j-1}\binom{m+t}{j}(n-j)^{m}.
\]
Similarly, summing over $x=1,\ldots n$ in the first display of Lemma 2.1 gives
\[
\Sigma n^m =\sum_{x=1}^n x^m=\sum_{j=0}^m (-1)^j\binom{2m+1}{j}\sum_{x=1}^n P_{m-j}(x).
\]
Defining the coefficients $C^P_{m\,n}=\sum_{x=1}^n P_{m}(x)$ (similarly $C^Q_{m\,n}=\sum_{x=1}^n Q_{m}(x)$) then gives us
\be
\Sigma n^m =\sum_{j=0}^m (-1)^j\binom{2m+1}{j}C^P_{(m-j)\,n}
\label{eq:c410}
\ee
where due to the underpinning polynomial definition of $C^P_{m\,n}$, for integer $t\geq 1$, the coefficients satisfy the recurrence relation
\be
C^P_{m\,n}=\sum_{j=1}^{m+t} (-1)^{j-1}\binom{m+t}{j}C^P_{m\,(n-j)}.
\label{eq:tm410}
\ee
We note that $(\ref{eq:c410})$ gives a fairly uneconomic way of writing Faulhaber's sum with integer coefficients. For example, taking $m=3$ and $n=5$ we have
\[
225=1^3+2^3+3^3+4^3 +5^3=855\binom{7}{0}-155\binom{7}{1}+30\binom{7}{2}-5\binom{7}{3}.
\]
However, akin to Faulhaber's $r$-fold sum notation, we can define
\[
\Sigma^0 C^P_{m\,n}=P_m(n),\qquad \Sigma^{r+1} C^P_{m\,n}=\Sigma^{r} C^P_{m\,1}+\Sigma^{r} C^P_{m\,2}+\ldots + \Sigma^{r} C^P_{m\,n},
\]
so that $(\ref{eq:c410})$ may be generalised for integer $r\geq 0$ to
\[
\Sigma^{r} n^m =\sum_{j=0}^m (-1)^j\binom{2m+1}{j}\Sigma^{r}C^P_{(m-j)\,n}.
\]
In fact this method allows us to generalise the $r$-fold sum of the first $m$ powers to a general polynomial $q(x)$. With the notation
$\Sigma q(n)^m =q(1)^m+q(2)^m+\ldots +q(n)^m$, we can define the
coefficients
\be
C^P_{m\,q(n)}=\sum_{x=1}^n P_{m}(q(x)),
\label{eq:psum}
\ee
so that
\be
\Sigma^{r} q(n)^m =\sum_{j=0}^m (-1)^j\binom{2m+1}{j}\Sigma^{r}C^P_{(m-j)\,q(n)}.
\label{eq:psum1}
\ee
For example, if $q(x)=x^2+x$, then we have
\[
\Sigma^{r} (n^2+n)^m=\sum_{j=0}^m (-1)^j\binom{2m+1}{j}\Sigma^{r}C^P_{(m-j)\,q(n)},
\]
which in the specific case when $m=2$ and $r=3$, gives
\[
\sum_{j=0}^2 (-1)^j\binom{5}{j}\Sigma^{3}C^P_{(2-j)\,q(n)}
=\frac{1}{70} \binom{n+2}{3} \left(2 n^4+22 n^3+119 n^2+369 n+818\right)
\]
\[
-5\times \frac{1}{10}  \binom{n+2}{3} \left(n^2+7 n+42\right) +10\times
 \binom{n+2}{3}
\]
\[
=\frac{4}{7}  \binom{n+4}{5} \left(n^2+4 n+2\right)=\Sigma^3 (n^2+n)^2.
\]
Taking $n=5$, we have that the 3-fold sum over the squares of the first 5 integers is 3384.
For $q(n)=n^2+n$, there are analogous results of $(\ref{eq:in6})$, such as $\Sigma^{r}C^P_{(m-j)\,n}$ being divisible by $\binom{n+r-1}{r}$.
We have just proved the following theorem.
\end{rmk}
\begin{thm1}
Let $q(x)\in\mathbb{Q}[x]$ be a polynomial of degree $d\geq 1$ in $x$ with non-constant term, and let the coefficients $C^P_{m\,q(n)}$ be defined as in $(\ref{eq:psum})$, with the notation
\[
\Sigma^0 C^P_{m\,q(n)}=P_m(q(n)),\qquad \Sigma^{r+1} C^P_{m\,q(n)}=\Sigma^{r} C^P_{m\,q(1)}+\Sigma^{r} C^P_{m\,q(2)}+\ldots + \Sigma^{r} C^P_{m\,q(n)}.
\]
Then the $r$th-fold sum over the first $n$ integers of the polynomial $q(x)$ to the $m$th power, denoted by $\Sigma^{r} q(n)^m$, is given by the relation $(\ref{eq:psum1})$.
\end{thm1}
\begin{lemma}
Our families of polynomials satisfy the recurrence relations
\be
x P_{m-1}(x)=Q_{m}(x)-Q_{m-1}(x),\qquad Q_{m}(x)=P_{m}(x)-P_{m-1}(x),
\label{eq:lm31}
\ee
\be
x\mathcal{P}_{m-1}(x)=\mathcal{Q}_{m}(x)-\mathcal{Q}_{m-1}(x),\qquad  \mathcal{Q}_{m}(x)=\mathcal{P}_{m}(x)-\mathcal{P}_{m-1}(x).
\label{eq:lm32}
\ee
\end{lemma}
\begin{corollary}
We have that
\be
P_{m+2}(x)=(x+2)P_{m+1}(x)-P_m(x),
\label{eq:lm321}
\ee
and
\be
P_m(x)=\sum_{j=0}^r (-1)^r\binom{r}{j}(x+2)^{r-j}P_{m-r-j}(x),
\label{eq:lm322}
\ee
where we note that similar relations hold for the polynomials $Q_m(x)$, $\mathcal{P}_m(x)$ and $\mathcal{Q}_m(x)$.
\end{corollary}

\begin{proof}
The relations $(\ref{eq:lm31})$ and $(\ref{eq:lm32})$ follow by direct manipulation of the binomial forms given in $(\ref{eq:s2})$, $(\ref{eq:s21})$, $(\ref{eq:s20})$, and $(\ref{eq:s210})$, from which we deduce $(\ref{eq:lm321})$ and $(\ref{eq:lm322})$.

\end{proof}

We now introduce some notation for falling and rising factorials.

\begin{defn}
For integers $r,k,$ with $k\geq 0$ let the rising factorial, falling factorial and hypergeometric functions be respectively defined in the usual manner such that
\[
r^{\overline{k}}=r(r+1)\ldots (r+k-1),\qquad r^{\underline{k}}=r(r-1)\ldots (r-k+1),
\]
and
\[
{}_mF_n\left ( \begin{array}{c|}a_1,\ldots,a_m\\
b_1,\ldots, b_n\end{array} \,\,z\right )=\sum_{k\geq 0} t_k,\qquad \mbox{where}\qquad
t_k=\frac{a_1^{\overline{k}}\ldots a_m^{\overline{k}}z^k}{b_1^{\overline{k}}\ldots b_m^{\overline{k}}k!},
\]
with none of the $b_i$ zero or a negative integer (to avoid division by zero).  

Then for $j,k,r,m,n\in\mathbb{Z}$, with $k,m,n\geq 0$, we have the identities
\be
r^{\underline{k}}\left (r-\frac{1}{2}\right )^{\underline{k}}
=\frac{(2r)^{\underline{2k}}}{2^{2k}},\qquad
\binom{r}{j}=(-1)^j\binom{j-r-1}{j},
\label{eq:b3}
\ee
\be
\binom{r-1/2}{r}=\left.\binom{2r}{r}\right /2^{2r},\qquad
(-1)^m\binom{-n-1}{m}=(-1)^n\binom{-m-1}{n}.
\label{eq:b4}
\ee
\end{defn}

\begin{lemma} Let
\[
\Sigma^0 C^{n P}_{m\,n}= n P_{m}(n),\qquad
\Sigma^{r+1} C^{n P}_{m\,n}=\Sigma^{r} C^{ P}_{m\,1}+\Sigma^{r} C^{2 P}_{m\,2}+\ldots + \Sigma^{r} C^{n P}_{m\,n},
\]
with a similar definition for $\Sigma^{r} C^{n Q}_{m\,n}$, so that by Lemma 2.2
\be
\Sigma^r C^{n P}_{m\,n}=\Sigma^r C^{Q}_{(m+1)\,n}+\Sigma^r C^{Q}_{m\,n},\qquad
\Sigma^r C^{ Q}_{m\,n}=\Sigma^r C^{P}_{m\,n}+\Sigma^r C^{P}_{(m-1)\,n}.
\label{eq:lm331}
\ee
Then there exist the \emph{inner polynomials} $p_m^{(r)}(x)$, $q_m^{(r)}(x)$, $u_m^{(r)}(x)$ and $v_m^{(r)}(x)$, each of degree $m$, with rational coefficients such that
\be
\Sigma^r C^{P}_{m\,n}=\binom{n+r-1}{r}\,p_m^{(r)}(n),\qquad \Sigma^r C^{Q}_{m\,n}=\binom{n+r-1}{r}\,q_m^{(r)}(n),
\label{eq:lm332}
\ee
\be
\Sigma^r C^{n P}_{m\,n}=\binom{n+r}{r+1}\,u_m^{(r)}(n),\qquad \Sigma^r C^{n Q}_{m\,n}=\binom{n+r}{r+1}\,v_m^{(r)}(n),
\label{eq:lm333}
\ee
with $p_m^{(r)}(x)$ and $q_m^{(r)}(x)$ both having leading term coefficient $\binom{m+r}{r}^{-1}$.
\end{lemma}
\begin{proof}
We give the proof for $p_m^{(r)}(x)$. When $r=0$, we have $p_m^{(0)}(x)=P_m(x)$, and applying Knuth's technique \cite{knuth} which relies on
the identity
\[
\sum_{t=1}^n t^{\overline{k}}=\frac{n^{\overline{k+1}}}{k+1},
\]
(proven later in a slightly different form in Lemma 4.1) we deduce that
\[
\sum_{t=1}^n\Sigma^r C^{P}_{m\,t}=\sum_{t=1}^n \binom{t+r-1}{r}\,p_m^{(r)}(t)=\binom{n+r}{r+1}\,p_m^{(r+1)}(n),
\]
for some polynomial $p_m^{(r+1)}(n)$. By our inductive assumption, $\binom{t+r-1}{r}\,p_m^{(r)}(t)$ is a polynomial of degree $m+r$ in $t$, and so by Faulhaber's formula, the sum from $1,2,\ldots n$ is a polynomial of degree $m+r+1$ in $n$. Now $\binom{n+r}{r+1}$ is a polynomial of degree $r+1$ in $n$ and so we must have that $p_m^{(r+1)}(n)$ is a polynomial of degree $m$ in $n$ for all integers $r\geq 0$.

The leading term normalisation coefficient $\binom{m+r}{r}$ also follows by this inductive approach and the proofs for $q_m^{(r)}(x)$, $u_m^{(r)}(x)$ and $v_m^{(r)}(x)$ are similar.
\end{proof}
A brief remark regarding the notation for our $r$-fold sum coefficients is that we could also write
\[
\Sigma^r C^{P}_{m\,n}=\Sigma^r P_m(n),\qquad \Sigma^r C^{n P}_{m\,n}=\Sigma^r n P_m(n),
\]
and similarly for the $Q$-based coefficients.
\begin{lemma}
With the polynomial definitions of the previous lemma, we have the inner polynomial recurrences
\be
q_m^{(r)}(x)=p_m^{(r)}(x)-p_{m-1}^{(r)}(x)
\label{eq:lm341}
\ee
\be
{\textstyle{\frac{n+r}{r+1}}}u_m^{(r)}(x)=q_{m+1}^{(r)}(x)-q_m^{(r)}(x)=p_{m+1}^{(r)}(x)-2 p_{m}^{(r)}(x)+p_{m-1}^{(r)}(x),
\label{eq:lm342}
\ee
and
\be
{\textstyle{\frac{n+r}{r+1}}}v_m^{(r)}(x)=q_{m+1}^{(r)}(x)-2 q_{m}^{(r)}(x)+q_{m-1}^{(r)}(x)
=p_{m+2}^{(r)}(x)-3 p_{m+1}^{(r)}(x)+3 p_{m}^{(r)}(x)-p_{m-1}^{(r)}(x).
\label{eq:lm343}
\ee

\end{lemma}

\begin{proof}
Substituting $(\ref{eq:lm332})$ and $(\ref{eq:lm333})$ of Lemma 2.3 in to $(\ref{eq:lm31})$ and $(\ref{eq:lm32})$ of Lemma 2.2 and then dividing through by $\binom{n+r-1}{r}$, we obtain
the above recurrence relations.
\end{proof}

\section{Polynomial Sums and Differences of $d$-th Powers}
With our families of polynomials established, we can now state some of their relations to the sum and difference of two $d$-th powers.

\begin{lemma}
Following the notation of T. Koshy (see \cite{koshy} ch 13), Let
\[
r=\frac{y+\sqrt{y^2+4x}}{2},\qquad s=\frac{y-\sqrt{y^2+4x}}{2},\qquad \Delta=\sqrt{y^2+4x},
\]
so that $r+s=y$, $-rs=x$, $2r=y+\Delta$ and $2s=y-\Delta$. Then for sums of powers of $r$ and $s$ we have that
\be
r^{2m+1}+s^{2m+1}=x^m y P_m\left (\frac{y^2}{x}\right )=y^{2m+1}P^{\rm Inv}_m\left (\frac{x}{y^2}\right ),
\label{eq:t101}
\ee
\be
r^{2m}+s^{2m}=x^m Q_m\left (\frac{y^2}{x}\right )=y^{2m}Q^{\rm Inv}_m\left (\frac{x}{y^2}\right ),
\label{eq:t102}
\ee

the differences of powers of $r$ and $s$ are given by
\be
r^{2m+1}-s^{2m+1}=\Delta\, x^m  \mathcal{P}_m\left (\frac{y^2}{x}\right )=\Delta\,
y^{2m}\,\mathcal{P}^{\rm Inv}_m\left (\frac{x}{y^2}\right ),
\label{eq:t103}
\ee
\be
r^{2m}-s^{2m}=\Delta \,\frac{x^{m}}{y} \mathcal{Q}_{m}\left (\frac{y^2}{x}\right )=\Delta\,y^{2m-1}\mathcal{Q}^{\rm Inv}_{m}\left (\frac{x}{y^2}\right ).
\label{eq:t104}
\ee
\end{lemma}
\begin{corollary}
\be
Q_n(x)^2-Q_{2n}(x)=2,\qquad Q_n(x)^2+Q_m(x)^2-Q_{n+m}(x)Q_{|n-m|}(x)=4,
\label{eq:lm33}
\ee
so that (suppressing the $x$ argument) we have
\be
(P_n-P_{n-1})^2+(P_m-P_{m-1})^2-(P_{n+m}-P_{n+m-1})(P_{|n-m|}-P_{|n-m|-1})=4,
\label{eq:lm34}
\ee
and for the $n=m$ case, with $n>0$, this yields
\be
2\left (P_n(x)-P_{n-1}(x)\right )^2-(P_{2n}(x)-P_{2n-1}(x))=Q_{2n}(x)+4.
\label{eq:lm35}
\ee
Similarly we find that for $n>m$
\be
x\left (P_n(x)^2+P_{m}(x)^2\right )=Q_{2n+1}(x)+Q_{2m+1}(x)-4,
\label{eq:lm36}
\ee
\be
x\left( P_{n+m}(x)P_{n-m}(x)\right )=Q_{2n+1}(x)- Q_{2m}(x),
\label{eq:lm37}
\ee
so that
\be
x\left (P_n(x)^2+P_{m}(x)^2-P_{n+m}(x)P_{n-m}(x)\right )=P_{2m+1}(x)-P_{2m-1}(x)-4.
\label{eq:lm38}
\ee
\end{corollary}
\begin{proof}
The relations $(\ref{eq:t101})$, $(\ref{eq:t102})$, $(\ref{eq:t103})$ and $(\ref{eq:t104})$, follow by combining the Fibonacci and Lucas polynomials expressions for $P_m(x)$, $Q_m(x)$, $\mathcal{P}_m(x)$, and $\mathcal{Q}_m(x)$, given previously in this section, with those for sums and differences of two $d$-powers derived in \cite{koshy}, Chapter 13.

For the Corollary, the first display in $(\ref{eq:lm33})$ is obtained by
taking $x=1$ in $(\ref{eq:t101})$, so that $rs=-1$, giving
\[
\left (Q_n(y^2)\right )^2-Q_{2n}(y^2)=\left (r^{2n}+s^{2n}\right )^2- \left (r^{4n}+s^{4n}\right )=(2r s)^{2n}=2,
\]
as required.

For $n>m>0$ and again taking $x=1$ in $(\ref{eq:t101})$, the left hand side of the second display in $(\ref{eq:lm33})$ simplifies to
\[
\left (Q_n(y^2)\right )^2+\left (Q_m(y^2)\right )^2-Q_{n+m}(y^2)Q_{n-m}(y^2)=
4+r^{4m}+s^{4m}-\left (\frac{r^{4m}+s^{4m}}{(rs)^{2m}}\right )=4.
\]
The case $m>n>0$ can be treated likewise, and hence the result.

The proofs for the relations $(\ref{eq:lm36})$, $(\ref{eq:lm37})$ and $(\ref{eq:lm38})$ follow in a similar fashion, where we note that
corresponding identities exist for the polynomials $\mathcal{P}_{n}(x)$ and $\mathcal{Q}_{n}(x)$.
\end{proof}

\begin{thm2}
Setting $x=(z^2-y^2)/4$ in Lemma 3.1, with $y$ and $z$ either both odd or both even positive integers, we have
\[
r=\frac{y+z}{2},\qquad s=\frac{y-z}{2},\qquad \Delta=z,
\]
so that $r$, $s$ and $\Delta$ each take integer values, and the expressions in $(\ref{eq:t101})$, $(\ref{eq:t102})$, $(\ref{eq:t103})$ and $(\ref{eq:t104})$,
all involve integer powers.

Furthermore, every integer value of $x$ can be attained. For odd $x=n$, we can take $y=n-1$ and $z=n+1$, so that $\Delta =n+1$, $r=n$ and $s=-1$, giving
\be
n^{2m+1}-1=\left (n-1\right )^{2m+1}P_m^{\text{\rm inv}}\left( \frac{n}{(n-1)^2}\right ),
\label{eq:tm11}
\ee
\be
n^{2m+1}+1=(n+1)\left (n-1\right )^{2m}\mathcal{P}_m^{\text{\rm inv}}\left( \frac{n}{(n-1)^2}\right ),
\label{eq:tm12}
\ee
\be
n^{2m}-1=(n+1)\left (n-1\right )^{2m-1}\mathcal{Q}_{m}^{\text{\rm inv}}\left( \frac{n}{(n-1)^2}\right ),
\label{eq:tm14}
\ee
\be
n^{2m}+1=\left (n-1\right )^{2m}Q_m^{\text{\rm inv}}\left( \frac{n}{(n-1)^2}\right ).
\label{eq:tm13}
\ee
Other integer producing values are $x=n(n+t)$ and $y=t$, yielding $r=n+t$ and $s=-n$.
\end{thm2}

\begin{proof}
If $y$ and $z$ are either both odd or both even then $x=(z^2-y^2)/4$ is an integer as squares of integers are either 0 or 1 modulo 4, depending on their parity. The eight identities then follow directly from Lemma 3.1.
\end{proof}
The polynomial structures considered here are fundamentally connected with hypergeometric functions involving binomial coefficients, and we now outline some of the binomial relations (not stated in \cite{concrete}) underpinning later results.

\begin{lemma}For each natural number $m$ we have
\be
(2m+1)\sum_{r=0}^k\frac{1}{2m+1-2(k+r)} \binom{2m-k-r}{k+r}\binom{k+r}{k-r}=\binom{2m+1}{2k},
\label{eq:b1}
\ee
\be
(2m+1)\sum_{r=0}^{k-1}\frac{1}{2m+1-2(k+r)} \binom{2m-k-r}{k+r}\binom{k+r}{k-r-1}=\binom{2m+1}{2k-1},
\label{eq:b2}
\ee
\be
(2m)\sum_{r=0}^k\frac{1}{2m-2(k+r)} \binom{2m-k-r-1}{k+r}\binom{k+r}{k-r}=\binom{2m}{2k},
\label{eq:be1}
\ee
and
\be
(2m)\sum_{r=0}^{k-1}\frac{1}{2m-2(k+r)} \binom{2m-k-r-1}{k+r}\binom{k+r}{k-r-1}=\binom{2m}{2k-1}.
\label{eq:be2}
\ee
\end{lemma}
\begin{proof} We give the proof for $(\ref{eq:b1})$.

Let
\[
\sum_{k\geq 0} t_k=\sum_{s=0}^k\frac{(2m+1)}{2m+1-2(k+s)} \binom{2m-k-s}{k+s}\binom{k+s}{k-s}.
\]
Then
\[
\frac{t_{k+1}}{t_k}=\frac{(k+s-m-1/2)(k+s-m)(k-s)}{(k+1/2)(k+s-2m)(k+1)}\times 1.
\]
We now recall Saalsch\"utz's identity, which states that for integer $n\geq 0$,
\be
F\left ( \begin{array}{c|}a,\,\,b,\,\,-n\\
c,\,a+b-c-n+1\end{array} \,\,1\right )=\frac{(a-c)^{\underline{n}}(b-c)^{\underline{n}}}{-c^{\underline{n}}(a+b-c)^{\underline{n}}},
\label{eq:b5}
\ee
and taking $a=s-m-1/2$, $b=s-m$, $c=1/2$ and $n=s$ in $(\ref{eq:b5})$ yields
\[
F\left ( \begin{array}{c|}s-m-1/2,\,\,s-m,\,\,-s\\
1/2,\,s-2m\end{array} \,\,1\right )=\frac{(s-m-1)^{\underline{s}}\,(s-m-1/2)^{\underline{s}}}
{\displaystyle{(-1/2)^{\underline{s}}}\,(2s-2m-1)^{\underline{s}}}.
\]
Hence we can read off the series as
\[
\sum_{k\geq 0} t_k=t_0 \,F\left ( \begin{array}{c|}s-m-1/2,\,\,s-m,\,\,-s\\
1/2,\,s-2m\end{array} \,\,1\right )=t_0\,\frac{(s-m-1)^{\underline{s}}\,(s-m-1/2)^{\underline{s}}}
{\displaystyle{(-1/2)^{\underline{s}}}\,(2s-2m-1)^{\underline{s}}},
\]
where
\[
t_0=\frac{(2m+1)}{(2m+1-2s)} \binom{2m-s}{s}.
\]
With the aid of the identities in $(\ref{eq:b3})$ and $(\ref{eq:b4})$ these falling factorial terms can be re-written such that
\[
(s-m-1)^{\underline{s}}=(-1)^s\,s!\,\binom{m}{s},\qquad (s-m-1/2)^{\underline{s}}=\frac{(2s-2m)^{\underline{2s}}}
{\displaystyle{2^{2s}}\,(s-m)^{\underline{s}}},
\]
\[
(s-m)^{\underline{s}}=(-1)^s\,s!\,\binom{m-1}{s},\qquad (2s-2m)^{\underline{2s}}=(2s)!\,\binom{2m-1}{2s},
\]
\[
\left (\frac{-1}{2}\right )^{\underline{s}}=\frac{(-1)^s (2s)!}{\displaystyle{2^{2s} s!}},\qquad
(2s-2m-1)^{\underline{s}}=(-1)^s s!\,\binom{2m-s}{s},
\]
and we have
\[
\sum_{k\geq 0} t_k=\frac{(2m+1)}{(2m+1-2s)} \binom{2m-s}{s}\frac{(s-m-1)^{\underline{s}}\,(s-m-1/2)^{\underline{s}}}
{\displaystyle{(-1/2)^{\underline{s}}}\,(2s-2m-1)^{\underline{s}}}
\]
\[
=\frac{(2m+1)\binom{2m-s}{s}(-1)^s s! \binom{m}{s}(2s)!\binom{2m-1}{2s}2^{2s}s!}
{(2m+1-2s)2^{2s}(-1)^s s!\binom{m-1}{s}(-1)^s(2s)!(-1)^s s!\binom{2m-s}{s}}
\]
\[
=\frac{(2m+1)\binom{m}{s}\binom{2m-1}{2s}}
{(2m+1-2s)\binom{m-1}{s}}
\]
\[
=\frac{(2m+1)m(2m-1)!}{(2m+1-2s)(m-s)(2m-1-2s)!(2s)!}=\binom{2m+1}{2s},
\]
as required.

The proofs for $(\ref{eq:b2})$, $(\ref{eq:be1})$ and $(\ref{eq:be2})$ are almost identical.
\end{proof}

\begin{thm3}For integers $a$ and $m$ with $m\geq 0$ we have
\be
P_m^{\rm Inv}(a(a+1)) = (a+1)^{2m+1}-a^{2m+1},\qquad (2a+1)\mathcal{Q}_m^{\rm Inv}(a(a+1)) = (a+1)^{2m}-a^{2m}
\label{eq:m48}
\ee
\be
(2a+1)\mathcal{P}_m^{\rm Inv}(a(a+1)) = (a+1)^{2m+1}+a^{2m+1},\qquad Q_m^{\rm Inv}(a(a+1)) = (a+1)^{2m}+a^{2m},
\label{eq:m481}
\ee
from which we deduce that for $0<b<a$,
\be
\sum_{j=b}^{a-1}P_m^{\rm Inv}(j(j+1)) = a^{2m+1}-b^{2m+1},\qquad \sum_{j=b}^{a-1}(2j+1)\mathcal{Q}_{m}^{\rm Inv}(j(j+1)) = a^{2m}-b^{2m}
\label{eq:m485}
\ee
\be
 \sum _{j=b}^{a-1} (-1)^{a+j-1}(2j+1)\mathcal{P}^{\rm Inv}_{m}(a(a+1))=a^{2m+1}+(-1)^{a+b-1}b^{2m+1},
\label{eq:m4815}
\ee
and
\be
 \sum _{j=b}^{a-1} (-1)^{a+j-1}Q^{\rm Inv}_m(a(a+1))=a^{2m}+(-1)^{a+b-1}b^{2m},
\label{eq:m4816}
\ee
Using the result that $\lim_{j\rightarrow 0}j^0(j+1)^0=1$ we can extend these results to $b=0$ so that for the expressions in $(\ref{eq:m485})$ we have
\be
\sum_{j=0}^{a-1}P_m^{\rm Inv}(j(j+1)) = a^{2m+1},\qquad \sum_{j=0}^{a-1}(2j+1)\mathcal{Q}_{m}^{\rm Inv}(j(j+1)) = a^{2m}.
\label{eq:m4850}
\ee

\end{thm3}
\begin{corollary1}Let $p=2m+1$ be an odd prime. Then an explicit form for the quotient $q$ in Fermat's ``Little Theorem'' $a^{p-1}=qp+1$, is given by
\be
q=\frac{a^{p-1}-1}{p}=\frac{1}{a}\sum_{k=1}^m \frac{1}{p-2k}\binom{p-k-1}{k}\sum_{j=1}^{a-1}j^k(j+1)^k \in \mathbb{N}
\label{eq:m52}
\ee
\[
=\frac{1}{a(2m+1)}\sum_{j=1}^{a-1}\left (\, _2F_1\left(-m-\frac{1}{2},-m;-2 m;-4 j (j+1)\right)-1\right ).
\]
\end{corollary1}
\begin{corollary2}For $a,b,c$ positive integers with $a\leq b\leq c$ and $m\geq 0$ we have $a^{2m+1} + b^{2m+1}-c^{2m+1}~=~0$ if and only if
\be
\sum_{j=0}^{a-1}P^{\rm Inv}_m(j(j+1))=\sum_{j=b}^{c-1}P^{\rm Inv}_m(j(j+1)).
\label{eq:m54}
\ee
Hence for odd exponents, Fermat's Last Theorem is equivalent to proving that equation $(\ref{eq:m54})$ has no solutions for $m\geq 1$.

Similarly, for integer $m\geq 1$ we have
\be
\sum_{j=0}^{a-1}(2j+1)\mathcal{Q}^{\rm Inv}_m(j(j+1))=\sum_{j=b}^{c-1}(2j+1)\mathcal{Q}^{\rm Inv}_m(j(j+1)),
\label{eq:m541}
\ee
so that for even exponents, Fermat's Last Theorem is equivalent to proving that equation $(\ref{eq:m541})$ has no solutions for $m\geq 2$.

\end{corollary2}
\begin{proof}

Setting $x=a^2+a$ and $y=1$ in Lemma 3.1, gives $\Delta = 2a+1$, $r=a+1$ and $s=-a$, and hence ($\ref{eq:m48}$) and ($\ref{eq:m481}$), from which the other relations of the Theorem can be deduced.

To see the first Corollary we make repeated use of the first identity in ($\ref{eq:m48}$) to obtain $a^{2m+1}=$
\[
1+\left((a+1)^{2m+1}-a^{2m+1}\right )+\left (a^{2m+1}-(a-1)^{2m+1}\right )+\ldots + \left (2^{2m+1}-1^{2m+1}\right )
\]
\be
=1+(2m+1)\sum_{k=0}^m \frac{1}{2(m-k)+1}\binom{2m-k}{k}\sum_{j=1}^{a-1}j^k(j+1)^k,
\label{eq:m50}
\ee
which can be written in the form
\be
a^{2m+1}-a=(2m+1)\sum_{k=1}^m \frac{1}{2(m-k)+1}\binom{2m-k}{k}\sum_{j=1}^{a-1}j^k(j+1)^k.
\label{eq:m51}
\ee
Noting that
\[
\frac{2m+1}{2(m-k)+1}\binom{2m-k}{k}\in \mathbb{N},\,\,\forall\hbox{ } 0\leq k \leq m,
\]
and using (\ref{eq:m51}) we have
\be
\frac{a^{2m}-1}{2m+1}=\frac{1}{a}\sum_{k=1}^m \frac{1}{2(m-k)+1}\binom{2m-k}{k}\sum_{j=1}^{a-1}j^k(j+1)^k,
\label{eq:m521}
\ee
where both sides are integers when $p=2m+1$ is a prime and $(a,p)=1$.

The second Corollary follows directly from ($\ref{eq:m485}$) and ($\ref{eq:m4816}$).
\end{proof}
For example, if $a=6$ and $m=2$ (so $p=2m+1=5$) in $(\ref{eq:m48})$, then
\[
5\sum_{k=0}^2 \binom{4-k}{k}\frac{6^k\times 7^k}{5-2k}
=5\left (\frac{1\times (42)^0}{5}+\frac{3\times(42)^1}{3}+\frac{1\times(42)^2}{1}\right )
=9031 =  7^5-6^5,
\]
and the equality in $(\ref{eq:m52})$ becomes $(6^4-1)/5=1295/5=1554/6=259=q$ so that $6^4=259\times 5 +1 $, or in congruence notation,
$ 6^4\equiv 1\pmod{5}.$

Taking $b=2$ in $(\ref{eq:m485})$ we obtain
\[
P_2^{\rm Inv}(6)+P_2^{\rm Inv}(12)+P_2^{\rm Inv}(20)+P_2^{\rm Inv}(30)=211+781+2101+4651=7433=6^5-2^5.
\]

\begin{rmk1}[to Theorem 3]
We note that in Theorem 3 a reciprocity law exists for negative values of $a$. If $a=-b$, with $b\geq 0$, then $(a+1)^{2m+1}-a^{2m+1}=b^{2m+1}-(b-1)^{2m+1}$. Telescoping as before we obtain the equivalence of $(\ref{eq:m4850})$, where the order of summation $j=-(b-1),-(b-2),\ldots,-2$, yields $(b-1)^{2m+1}$.

We also note how the binomial identities of Lemma 3.2 appear in $(\ref{eq:m48})$. Using the binomial expansion we obtain
\[
\sum_{k=0}^m \frac{2m+1}{2(m-k)+1} \binom{2m-k}{k}\sum_{s=0}^k\binom{k}{s}a^{k+s}
=\sum_{j=0}^{2m}\binom{2m+1}{j}a^j,
\]
and comparing coefficients of powers of $a$ yields
\be
\sum_{2t\geq j} \frac{2m+1}{2(m-t)+1} \binom{2m-t}{t}\binom{t}{j-t}=\binom{2m+1}{j}.
\label{eq:b10}
\ee
Setting $t=k+v$ in $(\ref{eq:b10})$ and then taking $j=2v$ or $j=2v-1$ respectively gives us $(\ref{eq:b1})$ or $(\ref{eq:b2})$ of Lemma 3.2 and so we have another proof
of the first identity in $(\ref{eq:m48})$.

\end{rmk1}

\begin{rmk2}[to Theorem 3]
A further consequence of $P_m^{\rm inv}(x^2+x)=(x+1)^{2m+1}-x^{2m+1}$, is that
\[
\sum_{x=1}^n \frac{1}{(x^2+x)^m}=\sum_{j=0}^m (-1)^j\binom{2m+1}{j}\sum_{x=1}^n \frac{(x+1)^{2m-2j+1}-x^{2m-2j+1}}{(x^2+x)^{m-j}},
\]
so that
\[
\frac{1}{(x^2+x)^m}=\sum_{k=0}^m (-1)^{m-k}\binom{2m+1}{m-k} \frac{(x+1)^{2k+1}-x^{2k+1}}{(x^2+x)^{k}}.
\]
\end{rmk2}

\begin{thm4}
For positive integers $m$ and $n$ we have
\be
\frac{2n+1}{2m+1}\in \mathbb{Z} \Leftrightarrow\frac{P_n(x)}{P_m(x)}\in \mathbb{Z}[x],\qquad \frac{P^{\rm Inv}_n(x)}{P^{\rm Inv}_m(x)}\in \mathbb{Z}[x],
\label{eqth41}
\ee
and
\be
\frac{n}{m}=2k+1 \text{ is odd } \Leftrightarrow \frac{Q_n(x)}{Q_m(x)}\in \mathbb{Z}[x],\qquad \frac{Q^{\rm Inv}_n(x)}{Q^{\rm Inv}_m(x)}\in \mathbb{Z}[x],
\label{eqth42}
\ee
from which we deduce that for $a\in \mathbb{N}$,
\be
\frac{2n+1}{2m+1}\in \mathbb{Z} \Leftrightarrow\frac{(a+1)^{2n+1}- a^{2n+1}}{(a+1)^{2m+1}- a^{2m+1}}\in \mathbb{N},
\label{eqth43}
\ee
and
\be
\frac{n}{m}=2k+1 \text{ is odd } \Leftrightarrow \frac{(a+1)^{2n}+ a^{2n}}{(a+1)^{2m}+ a^{2m}}\in \mathbb{N}.
\label{eqth44}
\ee
\end{thm4}

\begin{proof}
The divisibility properties of the polynomials $P_m(x)$, $P^{\rm Inv}$  $Q_m(x)$ and $Q^{\rm Inv}$ follow directly from their representations in terms of Fibonacci and Lucas polynomials given in $(\ref{eq:i18})$ and $(\ref{eq:i185})$ of Section 2. The integer divisibility properties stated in $(\ref{eqth43})$ and $(\ref{eqth44})$ then follow by setting $x=a^2+a$ in Theorems 2 and 3.
\end{proof}

\begin{rmk}[to Theorem 4]
An observation by M. N. Huxley is that setting $n=q m$, for any positive integer $q$, we can write a general form of the polynomial quotient stated
in $(\ref{eqth43})$ as
\be
\frac{(a+1)^{n}- a^{n}}{(a+1)^{m}- a^{m}}=\frac{(a+1)^{q m}- a^{q m}}{(a+1)^{m}- a^{m}}=\sum_{i=1}^q a^{(q-i)m}(a+1)^{(i-1)m},
\label{eqth45}
\ee
which is a geometric series with first term $a^{(q-1)m}$, constant ratio $((a+1)/a)^m$ and $q$-terms in total.

A similar general $q$-term geometric series exists for $(\ref{eqth44})$, with with first term $a^{(q-1)m}$, constant ratio $-((a+1)/a)^m$, and the added condition that $q$ is odd, so that we get
\be
\frac{(a+1)^{n}+ a^{n}}{(a+1)^{m}+ a^{m}}=\frac{(a+1)^{q m}+ a^{q m}}{(a+1)^{m}+ a^{m}}=\sum_{i=1}^q (-1)^{i-1}a^{(q-i)m}(a+1)^{(i-1)m}.
\label{eqth451}
\ee
\end{rmk}

\begin{thm5}
Let $u=n(n+1)$ and $v=\frac{n+1}{n^2}$. Then we may write
\[
4m\Sigma n^{2m-1}=\sum_{j=1}^m \binom{2m}{2j-1}P_{j-1}^{\rm Inv} (u)B_{2m-2j+1}+\sum_{j=1}^m \binom{2m}{2j}Q_{j}^{\rm Inv} (u)B_{2m-2j},
\]
\[
(4m+2)\Sigma n^{2m}=
\sum_{j=1}^{m+1} \binom{2m+1}{2j-1}(2n+1)\mathcal{P}_{j-1}^{\rm Inv} (u)B_{2m-2j+2}
\]
\[
+\sum_{j=1}^m \binom{2m+1}{2j}(2n+1)\mathcal{Q}_{j}^{\rm Inv} (u)B_{2m-2j+1},
\]
which we can sum over the first $n$ integers to get $4m\Sigma^2 n^{2m-1}=$
\[
\sum_{j=1}^m \binom{2m}{2j-1}n^{2j-1}P_{j-1}^{\rm Inv} (v)B_{2m-2j+1}+
\sum_{j=1}^m \binom{2m}{2j}\left ((n+2)n^{2j-1}Q_{j}^{\rm Inv} (v)+2\,\Sigma n^{2j} \right ) B_{2m-2j},
\]
and
\[
(4m+2)\Sigma^2 n^{2m}=
\sum_{j=1}^{m+1} \binom{2m+1}{2j-1}\left ( n^{2j-1}P_{j-1}^{\rm Inv} (v)+2\,\Sigma n^{2j-1} \right )B_{2m-2j+2}
\]
\[
+\sum_{j=1}^m \binom{2m+1}{2j}(n+2)n^{2j-1}\mathcal{Q}_{j}^{\rm Inv} (v)B_{2m-2j+1},
\]
\end{thm5}
\begin{corollary} We have
\[
\frac{1}{2m+1}\sum_{j=1}^{m+1}\binom{2m+1}{2j-1}\left (\Sigma^r n^{2j-1}\right ) B_{2m+2-2j}=\Sigma^{r+1} n^{2m}-\frac{1}{2}\Sigma^r n^{2m},
\]
and
\[
\frac{1}{2m}\sum_{j=1}^{m}\binom{2m}{2j}\left (\Sigma^r n^{2j}\right ) B_{2m-2j}=\Sigma^{r+1} n^{2m-1}-\frac{1}{2}\Sigma^r n^{2m-1},
\]
\end{corollary}

\begin{proof}
Setting $u=n(n+1)$, and using the Bernoulli polynomial identity $B_m(x+1)=(-1)^m B_m(-x)$, we have
\[
n+1=\frac{1+\sqrt{1+4u}}{2},
\]
and
\[
B_m\left (\frac{1+\sqrt{1+4u}}{2}  \right )=(-1)^m B_m\left (\frac{1-\sqrt{1+4u}}{2}  \right ),
\]
so that $m\Sigma n^{m-1}=$
\[
B_{m}\left (\frac{1+\sqrt{1+4u}}{2}  \right )-B_{m}(0)=(-1)^m B_{m}\left (\frac{1-\sqrt{1+4u}}{2}  \right )-B_{m}(0).
\]
Considering now $2m\Sigma n^{2m-1}$, we cancel the $B_{2m}$ coefficient in each sum and then add the sums to obtain
\[
4m\Sigma n^{2m-1}=\sum_{j=1}^{2m}\binom{2m}{j}\left [\left (\frac{1+\sqrt{1+4u}}{2}\right )^j + \left (\frac{1-\sqrt{1+4u}}{2}\right )^j\right ] B_{2m-j},
\]
which can be split into odd and even parts such that
\[
4m\Sigma n^{2m-1}=
\sum_{j=1}^{m}\binom{2m}{2j}\left [\left (\frac{1+\sqrt{1+4u}}{2}\right )^{2j} + \left (\frac{1-\sqrt{1+4u}}{2}\right )^{2j}\right ] B_{2m-2j}
\]
\[
+\sum_{j=1}^{m}\binom{2m}{2j-1}\left [\left (\frac{1+\sqrt{1+4u}}{2}\right )^{2j-1} + \left (\frac{1-\sqrt{1+4u}}{2}\right )^{2j-1}\right ] B_{2m-2j+1},
\]
and by Lemma 3.1 we have
\[
4m\Sigma n^{2m-1}=
\sum_{j=1}^{m}\binom{2m}{2j}  Q_j^{\rm Inv}(u) B_{2m-2j}+\sum_{j=1}^{m}\binom{2m}{2j-1}P_{j-1}^{\rm Inv}(u) B_{2m-2j+1},
\]
which is the first display of Theorem 5.

To see the second display we proceed as before to obtain
\[
(4m+2)\Sigma n^{2m}=
\sum_{j=1}^{m}\binom{2m+1}{2j}\left [\left (\frac{1+\sqrt{1+4u}}{2}\right )^{2j} - \left (\frac{1-\sqrt{1+4u}}{2}\right )^{2j}\right ] B_{2m-2j+1}
\]
\[
+\sum_{j=1}^{m+1}\binom{2m+1}{2j-1}\left [\left (\frac{1+\sqrt{1+4u}}{2}\right )^{2j-1} - \left (\frac{1-\sqrt{1+4u}}{2}\right )^{2j-1}\right ] B_{2m-2j+2},
\]
and by Lemma 3.1 we have $(4m+2)\Sigma n^{2m}=$
\[
\sum_{j=1}^{m}\binom{2m+1}{2j}  (2j+1)\mathcal{Q}_j^{\rm Inv}(u) B_{2m-2j+1}+\sum_{j=1}^{m}\binom{2m+1}{2j-1}(2j+1)\mathcal{P}_{j-1}^{\rm Inv}(u) B_{2m-2j+2},
\]
as required.

For the final two displays concerning the second-fold summations we apply Theorems 2 and 3 to obtain
\[
\sum_{n=1}^N P^{\rm Inv}_{j-1}(n(n+1))=(N+1)^{2j-1}-1=N^{2j-1}P^{\rm Inv}_{j-1}\left (\frac{N+1}{N^2}\right ),
\]
\[
\sum_{n=1}^N  Q^{\rm Inv}_j(n(n+1))=(N+1)^{2j}-1+2\,\Sigma N^{2j}
=(N+2)N^{2j-1}\mathcal{Q}^{\rm Inv}_{j}\left (\frac{N+1}{N^2}\right )+2\,\Sigma N^{2j},
\]
\[
\sum_{n=1}^N (2n+1)\mathcal{P}^{\rm Inv}_{j-1}(n(n+1))=(N+1)^{2j-1}-1+2\,\Sigma N^{2j-1}=N^{2j-1}P^{\rm Inv}_{j-1}\left (\frac{N+1}{N^2}\right )+2\,\Sigma N^{2j-1},
\]
\[
\sum_{n=1}^N  (2n+1)\mathcal{Q}^{\rm Inv}_j(n(n+1))=(N+1)^{2j}-1=(N+2)N^{2j-1}\mathcal{Q}^{\rm Inv}_{j}\left (\frac{N+1}{N^2}\right ).
\]
Summing over the first $n$ integers for the first two displays of Theorem 5 and then inserting the above four right hand side identities yields the required expressions.

To see the Corollary, instead of inserting the right hand expressions we use the central integer power minus one expressions. to obtain
$4m\Sigma^2 n^{2m-1}=$
\[
\sum_{j=1}^m \binom{2m}{2j-1}\left ((n+1)^{2j-1}-1\right )B_{2m-2j+1}+
\sum_{j=1}^m \binom{2m}{2j}\left ((n+1)^{2j}-1+2\,\Sigma n^{2j} \right ) B_{2m-2j},
\]
and
\[
(4m+2)\Sigma^2 n^{2m}=
\sum_{j=1}^{m+1} \binom{2m+1}{2j-1}\left ( (n+1)^{2j-1}-1+2\,\Sigma n^{2j-1} \right )B_{2m-2j+2}
\]
\[
+\sum_{j=1}^m \binom{2m+1}{2j}\left ((n+1)^{2j}-1\right )B_{2m-2j+1}.
\]
Combining the odd and even power sums in each of the two expressions, applying the identity $\sum_{j=1}^m \binom{m}{j}B_{m-j}=0$
and rearranging then gives us
\[
\frac{1}{2m+1}\sum_{j=1}^{m+1}\binom{2m+1}{2j-1}\left (\Sigma^1 n^{2j-1}\right ) B_{2m+2-2j}=\Sigma^{2} n^{2m}-\frac{1}{2}\Sigma^1 n^{2m},
\]
and
\[
\frac{1}{2m}\sum_{j=1}^{m}\binom{2m}{2j}\left (\Sigma^1 n^{2j}\right ) B_{2m-2j}=\Sigma^{2} n^{2m-1}-\frac{1}{2}\Sigma^1 n^{2m-1},
\]
where $\Sigma^1 n^m$ is the standard Faulhaber sum. Repeatedly summing ($r-1$) times over the first $n$ integers on both sides,
we deduce the results of the Corollary.
\end{proof}

\begin{rmk}[To the Corollary]
Let
\[
\Sigma^{r} n^{m}=\binom{n+r}{r+1}G_{m}^{(r)}(n),\quad \text{\rm with}\quad G_{m}^{(r)}(n)=\sum_{k=0}^{m-1} g_{m\, k}^{(r)} n^k,
\]
as in $(\ref{eq:in61})$, so that $G_{m}^{(r)}(n)$ is a polynomial in $n$ of degree $m-1$.

Then substituting for $G_{m}^{(r)}(n)$ into the displays of the Corollary, and cancelling the binomial coefficients yields the expressions
\be
\frac{1}{2m+1}\sum_{j=1}^{m+1}\binom{2m+1}{2j-1}G_{2j-1}^{(r)}(n) \,B_{2m+2-2j}=\frac{n+r+1}{r+2}G_{2m}^{(r+1)}(n)-\frac{1}{2}G_{2m}^{(r)}(n),
\label{eq:in71}
\ee
\be
\frac{1}{2m}\sum_{j=1}^{m}\binom{2m}{2j}G_{2j}^{(r)}(n)\, B_{2m-2j}=\frac{n+r+1}{r+2} G_{2m-1}^{(r+1)}(n)-\frac{1}{2}G_{2m-1}^{(r)}(n).
\label{eq:in81}
\ee
Equating coefficients of powers of $n$ in $(\ref{eq:in71})$ and rearranging then gives the Bernoulli relations between the $(r+1)$-fold and the the $r$-fold sum polynomial coefficients such that
\[
\sum_{k=0}^{2t}\binom{2m+1}{t}g_{(2m+1-k)\, (2m-2t)}^{(r)}B_{k}=\frac{2m+1}{r+2}\left (g_{(2m)\, (2m-1-2t)}^{(r+1)}+(r+1)g_{(2m)\, (2m-2t)}^{(r+1)}\right )
\]
and
\[
\sum_{k=0}^{2t}\binom{2m+1}{t}g_{(2m+1-k)\, (2m-1-2t)}^{(r)}B_{k}=\frac{2m+1}{r+2}\left (g_{(2m)\, (2m-2-2t)}^{(r+1)}+(r+1)g_{(2m)\, (2m-1-2t)}^{(r+1)}\right ).
\]
Similar identities follow from $(\ref{eq:in81})$, and using an analogous argument applied to the ``inner polynomials'' $p_m^{(r)}(x)$, $q_m^{(r)}(x)$, $u_m^{(r)}(x)$ and $v_m^{(r)}(x)$, each of degree $m$, as defined in Lemma 2.3, one can also deduce Bernoulli relations between the $(r+1)$-fold and the the $r$-fold sums in terms of these polynomial coefficients.
\end{rmk}


\section{Some Stirling Lemmas}

In the usual notation let $s(k,r)$ denote the Stirling numbers of the first kind,  $S(k,r)$ the Stirling numbers of the second kind (see \cite{concrete} 6.12), and let $m^{\ul k}$ denote the falling factorial
\[
m(m-1)(m-2)\ldots (m-k+1),
\]
so that $m^{\underline{k}}=(-1)^{k+1}m(1-m)_{k-1}$ in terms of the Pochhammer symbol $(a)_n$.

The $s(k,r)$ count the number of permutations of $k$ elements with $r$ disjoint cycles and are related to $m^{\ul k}$
by the identity
\be
m^{\ul k}=\sum_{r=1}^k s(k,r)m^r,
\label{eq:0}
\ee
whereas the $S(k,r)$ count the number of ways to partition a set of $k$ objects into $r$ non-empty subsets.

Expressions for $\Sigma n^m$ in terms of the Stirling numbers have been known for some time, such as the one given below in Lemma 4.1.
In this section we derive an expression for $\Sigma n^m$ in terms of the Stirling numbers,  leading to a Bernoulli-Stirling relation, before embarking
on a Stirling number path which connects back to the Faulhaber sum $\Sigma n^m$.

\begin{lemma}
We have
$$\Sigma n^m=\sum_{k=1}^m S(m,k){{(n+1)^{\underline{k+1}}} \over {k+1}}.$$
\end{lemma}

\begin{proof}
This follows from manipulation of the sum
$$\sum_{m=1}^n (m)_{k+1}={{(k+n+1)!} \over {(k+2)(n-1)!}},$$
the representation
$$j^m=\sum_{k=1}^m S(m,k)j^{\underline{k}},$$
and then reordering the double summation
$$\Sigma n^m=\sum_{j=1}^n \sum_{k=1}^m S(m,k)j^{\underline{k}}.$$
\end{proof}

\begin{lemma}
Let the integers $s_1(m,j)= s(m,j+1)+s(m,j)$.  Then
$$(m+1)\sum_{k=j}^m {{s(m,k)} \over {k+1}}(-1)^{k-j} {{k+1} \choose {k-j}}B_{k-j}=s_1(m,j).$$

\end{lemma}

\begin{proof}
We have the inverse relation for the falling factorial,
$$j^{\underline{m}}=(-1)^{m+1}j(1-j)_{m-1}=\sum_{k=1}^m s(m,k)j^k,$$
and may rewrite Faulhaber's formula with powers of $n$ as
$$\Sigma n^m={1 \over {m+1}}\sum_{k=1}^{m+1}(-1)^{m-k+1}{{m+1} \choose {m-k+1}}B_{m-k+1}n^k.$$
We then develop the double summation
$$\sum_{j=1}^n j^{\underline{m}}=\sum_{k=1}^m s(m,k)\sum_{j=1}^n j^k$$
$$=\sum_{\ell=1}^{m+1} \sum_{k=\ell-1}^m {{s(m,k)} \over {k+1}}(-1)^{k-\ell+1}{{k+1} \choose
{k-\ell+1}}B_{k-\ell+1}n^\ell.$$
The identification of the coefficients $s_1(r,j)$ is motivated by comparing with the
related generalised Stirling number triangle A094645 of the OEIS, which differ in signs.
It is then determined that $s_1(n,k)$ satisfies the recurrence for $n\geq k\geq 0$
$$s_1(n,k)=s_1(n-1,k-1)-(n-1)s_1(n-1,k),$$  
with $s_1(0,0)=s_1(1,0)=1$, and $s_1(0,1)=s_1(n,-1)=0$, consistent with the generating function of
columns $h_k(z)=(1+z)\ln^k(1+z)/k!$.

\end{proof}

\begin{thm6}
For integers $m\geq k\geq 0$, we have
$${{(-1)^{m-k}} \over {m+1}}{{m+1} \choose {m-k}}B_{m-k}=\sum_{\ell=k}^m {{S(m,\ell)} \over
{\ell+1}}[s(\ell,k+1)+s(\ell,k)].$$
\end{thm6}

\begin{proof}
From Lemma 4.2 we obtain
$$\sum_{j=1}^n j^{\underline{m}}=\sum_{\ell=0}^m \sum_{k=\ell}^m {{s(m,k)} \over {k+1}}(-1)^{k-\ell}{{k+1} \choose {k-\ell}}B_{k-\ell}n^{\ell+1}$$
$$={1 \over {m+1}}\sum_{\ell=0}^m s_1(m,\ell)n^{\ell+1}.$$
Then for Faulhaber's sum $\Sigma n^m$ we obtain
$$\Sigma n^m=\sum_{k=1}^m S(m,k)\sum_{j=1}^n j^{\underline{k}}$$
$$=\sum_{k=1}^m {{S(m,k)} \over {k+1}}\sum_{\ell=0}^k s_1(k,\ell)n^{\ell+1}$$
$$=\sum_{\ell=0}^m \sum_{k=\ell}^m {{S(m,k)} \over {k+1}} s_1(k,\ell)n^{\ell+1},$$
and equating the coefficients of $n^{k+1}$ in this expression with those of the standard Bernoulli sum for Faulhaber's
formula, the identity follows.
\end{proof}

Returning now to the relation ($\ref{eq:0}$), replacing $m$ with $m+i$ yields
\[
(m+i)^{\ul k}=(m+i)(m+i-1)(m+i-2)\ldots(m+i-k+1)
\]
\[
=\sum_{r=1}^k s(k,r)(m+i)^r =\sum_{r=1}^k s(k,r)\sum_{j=0}^r \binom{r}{j}m^j i^{r-j},
\]
and collecting terms we can write
\[
(m+i)^{\ul k}=\sum_{r=0}^k m^r\sum_{t=0}^{k-r}\binom{r+t}{r}s(k,r+t)i^t.
\]
Considering the symmetric case $(m+k)^{\ul{2k}}$
\be
=(m+k)(m+k-1)\ldots(m-k+1)=\sum_{r=0}^{2k} m^r\sum_{t=0}^{2k-r}\binom{r+t}{r}s(2k,r+t)k^t,
\label{eq:1}
\ee
we now show that it is possible to express $(\ref{eq:1})$ as an integer coefficient polynomial in $(m(m+1))$ of degree $k$.

Let $c(k,r)$, be the coefficient of $(m(m+1))^r$, so that
\be (m+k)^{\ul{2k}} = \sum_{r=1}^k c(k,r)(m(m+1))^r =
\sum_{r=1}^k c(t,r) m^r \sum_{s=0}^r \binom{r}{s}m^s.
\label{eq:2}
\ee
Then equating for powers of $m$ in $(\ref{eq:2})$, we find that the respective coefficients of
$m^{2k-2r}$ and $m^{2k-2r-1}$  are given by
\be \sum_{i=0}^{r} c(k,k-i) \binom{k-i}{k-2r+i},
\label{eq:3}
\ee
and
\be
\sum_{i=0}^{r} c(k,k-i) \binom{k-i}{k-2r-1+i},
\label{eq:4}
\ee
In both cases, equating $(\ref{eq:3})$
and $(\ref{eq:4})$ with $(\ref{eq:1})$ yields the identity
\be
\sum_{i=0}^{r}\binom{k-i}{2r-2i}c(k,k-i)=\sum_{i=0}^{2r}\binom{2k-2r+i}{i}s(2k,2k-2r+i)k^i,
\label{eq:5}
\ee
which relates the $c(k,r)$ numbers to the Stirling numbers of the first kind. It also enables us to write
\be c(k,r)=
-\sum_{i=r+1}^{k}\binom{i}{2i-2r}c(k,i)+\sum_{i=0}^{2k-2r}\binom{2r+i}{i}s(2k,2r+i)k^i,
\label{eq:6}
\ee
where
$c(k,k)=1$. Hence $c(k,r) \in \mathbb{Z}$ for all $1\leq r \leq k$ and $(m+k)^{\ul 2k}$ can be expressed as an integer coefficient polynomial in $(m(m+1))$ with coefficients in $\mathbb{Z}$. The first few coefficients $c(k,r)$ are given
below.
\[
\begin{array}{|c||c|c|c|c|c|c|c|c|}
\hline k \backslash r &r=0&r= 1 &r= 2 &r= 3 &r= 4 &r= 5 & r=6 & r=7 \\ \hline\hline k=0&1 & 0 &0 &0 &0 &0 &0 &0 \\ \hline k=1&0 & 1 &0 &0 &0 &0 &0 &0 \\ \hline k=2& 0 & -2 & 1 &0 &0 &0 &0 &0 \\ \hline k=3&0 & 12 & -8 & 1 &0 &0 &0
&0 \\ \hline k=4& 0 & -144 & 108 & -20 & 1 &0 &0 &0 \\ \hline k=5& 0 & 2880 & -2304 & 508 & -40 & 1 &0 &0 \\ \hline k=6&0 & -86400 & 72000 & -17544 & 1708 & -70
& 1 &0 \\  \hline k=7& 0 & 3628800 & -3110400  & 808848 & -89280& 4648 & -112 & 1\\ \hline
\end{array}
\]
Some special values are
\[
c(k,k-1)=-2\binom{k+1}{3},\qquad c(k,1)=(-1)^{k+1}k!(k-1)!,
\]
and
\[
c(k,2)=(-1)^k (k-1)!\left (k!-(k-1)!\right ).
\]
The integers $c(k,r)$ are in fact the \emph{Legendre-Stirling numbers of the first kind}, $Ps_k^{(r)}$, the main properties of which we state in the following lemmas.

\begin{lemma}
Let $\langle m \rangle^{\ul k}$ denote the generalised falling factorial symbol defined such that
\[
\langle m \rangle^{{\ul k}}= \prod_{r=0}^{k-1}(m-r(r+1)),\qquad (\langle m \rangle^{\ul 0}=1).
\]
Then the Legendre-Stirling numbers of the first kind, $Ps_k^{(r)}$, satisfy the horizontal generating function
\be
\langle m \rangle^{\ul k} = \sum_{r=0}^{k}Ps_{k}^{(r)} m^r.
\label{eq:55}
\ee
\end{lemma}

\begin{lemma}
The Legendre-Stirling numbers of the first kind, $Ps_k^{(r)}$ satisfy the following triangular recurrence relation
\be
Ps_k^{(r)}=Ps_{k-1}^{(r-1)}-k(k-1)Ps_{k-1}^{(r)},
\label{eq:56}
\ee
where $Ps_k^{(0)}=Ps_0^{(r)}=0$, except $Ps_0^{(0)}=1$.
\end{lemma}

\begin{lemma} The Legendre-Stirling numbers of the first kind, $Ps_k^{(r)}$, satisfy the following identities
\be
\sum_{r=0}^k (-1)^{k+r}Ps_k^{(r)}t^{k-r}=\prod_{r=0}^{k-1} \left (1+r(r+1)t\right ),
\label{eq:57}
\ee
and
\be
\sum_{r=0}^k Ps_k^{(r)}t^{k-r}=\prod_{r=0}^{k-1} \left (1-r(r+1)t\right ).
\label{eq:58}
\ee
\end{lemma}
\begin{proof}
For proofs of the above three lemmas and other properties relating to the Legendre-Stirling numbers we refer the reader to $\cite{andrews2}$.
\end{proof}

\begin{lemma}
The coefficients $c(k,r)$ of $(m(m+1))^r$ in the expansion of $(m+k)^{{\ul{2k}}}$ given in $(\ref{eq:2})$ are the Legendre-Stirling numbers of the first kind $Ps_k^{(r)}$. Hence $c(k,r)=Ps_k^{(r)}$ and we can replace $Ps_k^{(r)}$ with $c(k,r)$ in the above three lemmas.
\end{lemma}
\begin{corollary}
The Legendre-Stirling numbers of the first kind, $Ps_k^{(r)}$, satisfy the recurrence relation for $c(k,r)$ given in $(\ref{eq:6})$.
\end{corollary}
\begin{proof}
To prove the equality of $c(k,r)$ and $Ps_k^{(r)}$ we need only show that
\[
(m+k)^{\ul{2k}}=\prod_{r=0}^{2k-1}(m+k-r)=\langle m(m+1) \rangle^{\ul k}=\prod_{r=0}^{k-1}(m(m+1)-r(r+1)).
\]
Now $(m+k)^{\ul{2k}}$
\[
=(m+k)(m+k-1)\ldots(m+2)(m+1)m(m-1)\ldots(m-(k-2))(m-(k-1)),
\]
\[
=(m+k)(m-(k-1)\ldots (m+k-r)(m-(k-r-1))\ldots (m+2)(m-1)(m+1)m.
\]
\[
=\prod_{r=0}^{k-1}(m(m+1)-r(r+1))=\langle m(m+1) \rangle^{\ul k}\,,
\]
and hence the result.

The Corollary then follows by replacing $c(k,r)$ with $Ps_k^{(r)}$ and $c(k,i)$ with $Ps_k^{(i)}$
in $(\ref{eq:6})$.
\end{proof}

\begin{rmk}
Comparing the horizontal generating function for the Stirling numbers of the first kind in $(\ref{eq:0})$, with
that for the Legendre-Stirling numbers of the first kind in $(\ref{eq:55})$, we see that both types of numbers satisfy a similar formula. This similarity between the two types of numbers also extends to the triangular recurrence given in $(\ref{eq:56})$, which is the counterpart to the classical triangular identity for the Stirling numbers of the first kind
\[
s_k^{(r)}=s_{k-1}^{(r-1)}-(k-1)s_{k-1}^{(r)}.
\]
Further symmetries concern the existence of the Legendre-Stirling numbers of the second kind $PS_k^{(r)}$, which obey similar relations to their classical cousins the Stirling numbers of the second kind. The properties of these numbers are examined in the paper by Andrews, Gawronski and Littlejohn \cite{andrews2}. We give the first few $PS_k^{(r)}$ numbers in the triangular array below, where we note that this is the inverse matrix of the triangular array for $Ps_k^{(r)}$.
\[
\begin{array}{|c||c|c|c|c|c|c|c|c|}
\hline k \backslash r &r=0&r= 1 &r= 2 &r= 3 &r= 4 &r= 5 & r=6 & r=7 \\
\hline\hline k=0&1 & 0 &0 &0 &0 &0 &0 &0 \\
\hline k=1& 0 & 1 & 0 & 0 & 0 & 0 & 0 & 0 \\
\hline k=2& 0 & 2 & 1 & 0 & 0 & 0 & 0 & 0 \\
\hline k=3&0 & 4 & 8 & 1 & 0 & 0 & 0 & 0 \\
\hline k=4& 0 & 8 & 52 & 20 & 1 & 0 & 0 & 0 \\
\hline k=5& 0 & 16 & 320 & 292 & 40 & 1 & 0 & 0 \\
\hline k=6& 0 & 32 & 1936 & 3824 & 1092 & 70 & 1 & 0 \\
\hline k=7& 0 & 64 & 11648 & 47824 & 25664 & 3192 & 112 & 1 \\
\hline
\end{array}
\]
It is also worth noting that there are combinatorial interpretations of the Legendre-Stirling numbers of both kinds which are outlined in \cite{andrews1} and \cite{egge}.

The identity $(\ref{eq:5})$ directly relates the Legendre-Stirling numbers of the first kind to the Stirling numbers of the first kind. Not surprisingly this ``meeting point'' between the two types of numbers is of interest and we summarise the main characteristics of the sums involved in the following lemma.
\end{rmk}
\begin{lemma}
Let $\mathcal{S}^{(r)}_k$ denote the sum in $(\ref{eq:5})$, so that in the orientation of $(\ref{eq:6})$ we write
\be
\mathcal{S}^{(r)}_k=\sum_{i=r}^{k}\binom{i}{2i-2r}Ps_k^{(i)}=\sum_{i=0}^{k}\binom{i}{2i-2r}Ps_k^{(i)}.
\label{eq:795}
\ee
Then the numbers $\mathcal{S}^{(r)}_k$ are the integer sequence (OEIS A204579) and so satisfy the triangular recurrence relation
\be
\mathcal{S}^{(r)}_k=-(k-1)^2\mathcal{S}_{k-1}^{(r)}+\mathcal{S}_{k-1}^{(r-1)}.
\label{eq:80}
\ee
If in addition we define $\mathcal{S}_k$ and $\mathcal{S}_k^+$ such that
\[
\mathcal{S}_k=\sum_{r=0}^{k}\mathcal{S}_k^{(r)},\qquad \mathcal{S}^+_k=\sum_{r=0}^{k}|\mathcal{S}_k^{(r)}|,
\]
then $\mathcal{S}_0=\mathcal{S}_1=\mathcal{S}^+_0=\mathcal{S}^+_1=1$.
Considering the ratio $\mathcal{S}_k^+/\mathcal{S}_{k-1}^+$ we have
\[
\mathcal{S}_k^+ / \mathcal{S}_{k-1}^+=1+\sum_{t=1}^{k-1}(2t-1)=1+(k-1)^2,
\]
from which we deduce that for $k\geq 1$,
\be
\mathcal{S}_k^+=\prod_{t=1}^{k}(1+(t-1)^2)=\prod_{t=0}^{k-1}(1+t^2),\qquad
\mathcal{S}_k=\prod_{t=0}^{k-1}(1-t^2),
\label{eq:81}
\ee
so that for $k\geq 2$, $\mathcal{S}_k=0$.
It also follows that a generating function for the $\mathcal{S}_k^{(r)}$ is given by
\be
\prod_{t=0}^{k-1}(x-t^2)=\sum_{r=1}^k\mathcal{S}_k^{(r)}x^{r}.
\label{eq:815}
\ee
\end{lemma}
\begin{proof}
The recurrence relation follows by substituting the expression for $\mathcal{S}_k^{(r)}$ in $(\ref{eq:795})$ into the identity $(\ref{eq:80})$ and rearranging to establish that the triangular array of integers is in fact the integer sequence OEIS A20459. This sequence is a signed version of A008955, with rows in reverse order, where A008955 is the triangle of \emph{central factorial numbers}.
\end{proof}
We give the first few $\mathcal{S}_k^{(r)}$ in the table below.
\[
\begin{array}{|c||c|c|c|c|c|c|c|c|}
\hline k \backslash r&r=0&r= 1 &r= 2 &r= 3 &r= 4 &r= 5 & r=6 & r=7 \\
\hline\hline  k=0 &1 & 0 & 0 & 0 & 0 & 0 & 0 & 0 \\
\hline k=1 &0 & 1 & 0 & 0 & 0 & 0 & 0 & 0 \\
\hline k=2 & 0 & -1 & 1 & 0 & 0 & 0 & 0 & 0 \\
\hline k=3 & 0 & 4 & -5 & 1 & 0 & 0 & 0 & 0 \\
\hline k=4 & 0 & -36 & 49 & -14 & 1 & 0 & 0 & 0 \\
\hline k=5 & 0 & 576 & -820 & 273 & -30 & 1 & 0 & 0 \\
\hline k=6 & 0 & -14400 & 21076 & -7645 & 1023 & -55 & 1 & 0 \\
\hline k=7 & 0 & 518400 & -773136 & 296296 & -44473 & 3003 & -91 & 1 \\ \hline
\end{array}
\]
From the recurrence in $(\ref{eq:80})$ it follows that $\mathcal{S}_k^{(k)}=1$, and
\[
\mathcal{S}_k^{(k-1)}=-\frac{(2k-1)}{3}\binom{k}{2},\qquad \mathcal{S}_k^{(1)}=(-1)^{k-1}\left ((k-1)!\right )^2.
\]
Another identity listed in the OEIS for the integer sequence A204579 is given by
\[
\mathcal{S}_k^{(r)}=\sum_{i=0}^{2r}(-1)^{k+i}s(k,i)\,s(k,2r-i),
\]
where as previously defined, $s(k,r)$ are the Stirling numbers of the first kind. Given the similarities between the $\mathcal{S}_k^{(r)}$ and both the Stirling numbers of the first kind and the Legendre-Stirling numbers of the first kind, we refer to the numbers $\mathcal{S}_k^{(r)}$ as the \emph{generalised Stirling numbers of the first kind}.

It is demonstrated by Andrews et al. \cite{andrews2} that the inverse of the lower triangular matrix whose entries are the Legendre-Stirling numbers of the second kind is the lower triangular matrix whose entries are the Legendre-Stirling numbers of the first kind (this inverse matrix relation also holds between the classical Stirling numbers of both kinds). The theory underlying these inverse characterisations is outlined in the paper by Milne and Bhatnagar \cite{milne}.

Applying this theory to the previous triangular array, let $\mathcal{T}_k^{(r)}$ denote the entries of the inverse table for the $\mathcal{S}_k^{(r)}$, given below.
\[
\begin{array}{|c||c|c|c|c|c|c|c|c|}
\hline k \backslash r&r=0&r= 1 &r= 2 &r= 3 &r= 4 &r= 5 & r=6 & r=7 \\
\hline\hline  k=0 & 1 & 0 & 0 & 0 & 0 & 0 & 0 & 0 \\
\hline k=1 & 0 & 1 & 0 & 0 & 0 & 0 & 0 & 0 \\
\hline k=2 & 0 & 1 & 1 & 0 & 0 & 0 & 0 & 0 \\
\hline k=3 & 0 & 1 & 5 & 1 & 0 & 0 & 0 & 0 \\
\hline k=4 & 0 & 1 & 21 & 14 & 1 & 0 & 0 & 0 \\
\hline k=5 & 0 & 1 & 85 & 147 & 30 & 1 & 0 & 0 \\
\hline k=6 & 0 & 1 & 341 & 1408 & 627 & 55 & 1 & 0 \\
\hline k=7 & 0 & 1 & 1365 & 13013 & 11440 & 2002 & 91 & 1 \\
\hline
\end{array}
\]
Then the $\mathcal{T}^{(r)}_k$ satisfy the recurrence relation
\be
\mathcal{T}^{(r)}_k=r^2\mathcal{T}_{k-1}^{(r)}+\mathcal{T}_{k-1}^{(r-1)}.
\label{eq:82}
\ee
and produce the integer sequence of \emph{central factorial numbers of the second kind}, listed as A036969 in the OEIS. These numbers are also referred to in the OEIS entry as \emph{generalised Stirling numbers of the second kind}.

Having traversed from the Stirling numbers through the Legendre-Stirling numbers to the generalised Stirling numbers, we have now come ``full circle'' in terms of Faulhaber's sum $\Sigma n^m$ and the Bernoulli numbers $B_k$. With the established notation for the generalised Stirling numbers of the second kind (or integer central factorial numbers), it is shown by Knuth \cite{knuth} that
\[
\Sigma n^{2m-1} = \sum_{k=1}^m (2k-1)! \mathcal{T}_m^{(k)}\binom{n+k}{2k},
\]
and
\[
\Sigma n^{2m} = \sum_{k=1}^m (2k-1)! \mathcal{T}_m^{(k)}\frac{k\,(2n+1)}{2k+1}\binom{n+k}{2k}.
\]
For example
\[
\Sigma n^7=5040\binom{n+4}{8}+1680\binom{n+3}{6}+126\binom{n+2}{4}+\binom{n+1}{2},
\]
and so these two formula give integer coefficient expansions for the Faulhaber sums~$\Sigma n^m$.

A related observation is that the generalised Stirling numbers of the second kind $\mathcal{T}_{k}^{(r)}$, can be used to express the $2k$th Bernoulli number $B_{2k}$, such that
\[
B_{2k}=\frac{1}{2}\sum_{r=1}^k(-1)^{r+1}\frac{(r-1)!r!}{2r+1}\mathcal{T}_{k}^{(r)}.
\]
To conclude this section, we now give some more detail on the Legendre-Stirling numbers of the second kind $PS_k^{(r)}$, defined by $PS_k^{(0)}=PS_0^{(r)}=0,$ except $PS_0^{(0)}=1$, and by
\[
PS_k^{(r)}=\sum_{i=1}^r
(-1)^{i+k}\frac{(2i+1)(i^2+i)^k}{(i+r+1)!(r-i)!}.
\]
The properties of these numbers are examined in  \cite{andrews2}, wherein it was shown that
\[
x^n=\sum_{r=0}^n PS_n^{(r)}\langle x \rangle^{\ul r},
\]
so that
$\{\langle x \rangle^{\ul r}\}_{r=0}^{\infty}$ forms a basis for the vector space
of polynomials over some field, under the usual operations of addition and scalar multiplication.

One possible point of interest is that a vertical exponential generating function for the Legendre-Stirling numbers of the second kind is given by
$$\phi_r(t)=\sum_{n=0}^\infty {{PS_n^{(r)}} \over {n!}}t^n=\sum_{j=0}^r (-1)^{r+j} {{(2j+1)e^{j(j+1)t}} \over
{(r+j+1)!(r-j)!}}$$
$$={1 \over {(2r)!}}\sum_{m=0}^{2r} (-1)^m {{2r} \choose m} e^{(r-m)(r+1-m)t},$$
and it is evident that $\phi_r(0)=PS_0^{(r)}=0$.

Within Remark 5.3 of \cite{andrews2} it is noted that $\phi_r(t)$ has the form
$$\phi_r(t)={1 \over {(2r)!}}(e^{2t}-1)^r p_r(e^{2t}),$$
wherein $p_r(x)$ is a certain polynomial of degree $r(r-1)/2$.  These polynomials are not further described in
\cite{andrews2}.  However, we give an explicit expression for them.

The polynomials $p_r(x)$ have as powers only triangular numbers $i(i+1)/2$, satisfy $p_r(1)=0$ for all $r>0$,
$(2r)!p_r(-1)=-2^r$, and $(2r)!p_r(0)=(-1)^r {{2r} \choose r}/(r+1)$, the signed Catalan numbers.  The first few
of these polynomials are given in the following list.
\begin{table}[h]
\centering
\label{phipolys}
\begin{tabular}{cc}
$j$ & $(2j)!p_j(x)$\\ \hline
0 &  $1$ \\
1 &  $x-1$\\
2 &  $x^3-3x+2$\\
3 & $x^6-5x^3+9x-5$\\
4 & $x^{10}-7x^6+20x^3-28x+14$\\
5 & $x^{15}-9x^{10}+35x^6-75x^3+90x-42$\\
6 & $x^{21}-11x^{15}+54x^{10}-154x^6+275x^3-297x+132$\\
\end{tabular}
\end{table}
Introducing the coefficients
$$g(k,i)={{(k-2i+1)} \over {(k-i+1)}}{k \choose i}$$
enables us to write
$$p_r(x)=(-1)^r\sum_{i=0}^r g(2r,r-i)x^{i(i+1)/2}.$$

\section{Binomial Expansions and Interpolating Polynomials}
In this final section, we use the previously developed theory to examine some binomial decompositions and obtain interpolating polynomials for integer powers of real numbers. In doing so we find, as with Faulhaber's sums $\Sigma^r n^m$ that the even and odd power interpolating polynomials have different ``shapes''. Specifically the even power interpolating polynomials have the form 1 plus a polynomial in $x(x+1)$, whereas for the odd power interpolating polynomials have the form 1 plus a polynomial in the odd powers of $x$, commencing $x^{-1}, x^1, x^3, x^5,\ldots$.
\begin{lemma}
With $Ps_k^{(r)}$ the Legendre-Stirling numbers of the first kind and $\mathcal{S}_k^{(r)}$ the generalised Stirling numbers of the first kind, we can express the binomial coefficients $\binom{m+k}{2k}$ and $\binom{m+k-1}{2k-1}$ such that
\be
\binom{m+k}{2k} =\frac{\langle m(m+1) \rangle^{\ul k}}{(2k)!}=\frac{\left (m+k \right )^{\ul{2k}}}{(2k)!}= \frac{1}{(2k)!}\sum_{r=0}^k Ps_k^{(r)}(m(m+1))^r.
\label{eq:67}
\ee
and
\be
\binom{m+k-1}{2k-1}=m\prod_{t=1}^{k-1}(m^2-t^2) =\frac{\left ( m+k-1 \right )^{\ul{2k-1}}}{(2k-1)!}= \frac{1}{(2k-1)!}\sum_{r=1}^k \mathcal{S}_k^{(r)}m^{2r-1}.
\label{eq:675}
\ee
\end{lemma}
For example, if $k=4$, then
\[
\binom{m-4+8}{8} =\frac{\langle m(m+1) \rangle^{\ul 4}}{8!}= \frac{1}{8!}\sum_{r=0}^4 Ps_4^{(r)}(m(m+1))^r
\]
\be =\frac{1}{8!}\left ( (m(m+1))^4 -20(m(m+1))^3+108(m(m+1))^2-144(m(m+1)) \right ).
\label{eq:600}
\ee
\begin{corollary} For $m\geq 1$ a positive integer, we have
\be
P_m^{\rm Inv}(x)=\sum_{k=0}^m\frac{2m+1}{(2k+1)!}\left (\sum_{r=0}^k Ps_k^{(r)}(m(m+1))^r \right )x^{m-k},
\label{eq:thm71}
\ee
and
\be
Q_m^{\rm Inv}(x)=\sum_{k=0}^m\frac{2m}{(2k)!}\left (\sum_{r=0}^k \mathcal{S}_k^{(r)}m^{2r-1}\right )x^{m-k},
\label{eq:thm712}
\ee
so that for $m\geq 0$, we may write
\be
a^{2m+1}=\sum_{j=0}^m P_m^{\rm Inv}(j(j+1)) =1+\sum_{k=0}^m \frac{2m+1}{(2k+1)!}\sum_{r=0}^k Ps_k^{(r)}(m(m+1))^r\sum_{j=1}^{a-1}(j^2+j)^{m-k}
\label{eq:66}
\ee
\be
=1+(2m+1)\sum_{r=0}^m (m(m+1))^r\sum_{k=r}^m \frac{Ps_k^{(r)}}{(2k+1)!}\sum_{j=1}^{a-1}(j^2+j)^{m-k},
\label{eq:661}
\ee
and for $m\geq 1$
\be
a^{2m}=(-1)^{a-1}+\sum_{k=0}^m \frac{2m}{(2k)!}\sum_{r=0}^k \mathcal{S}_k^{(r)}m^{2r-1}\sum_{j=1}^{a-1}(-1)^{a+j-1}(j^2+j)^{m-k}
\label{eq:665}
\ee
\be
=(-1)^{a-1}+(2m)\sum_{r=0}^m m^{2r-1}\sum_{k=r}^m \frac{\mathcal{S}_k^{(r)}}{(2k)!}\sum_{j=1}^{a-1}(-1)^{a+j-1}(j^2+j)^{m-k}.
\label{eq:667}
\ee
\end{corollary}
\begin{proof}[Proof of Corollary]
Reversing the order of summation for $k$ in $(\ref{eq:m50})$ gives
\[
a^{2m+1}=1+(2m+1)\sum_{k=0}^m \frac{1}{2k+1}\binom{m+k}{2k}\sum_{j=1}^{a-1}(j^2+j)^{m-k},
\]
and substituting for the binomial coefficient with $(\ref{eq:67})$, we obtain the identity $(\ref{eq:66})$. Noting that an $r$ value occurs for every $k$ value with $k\geq r$, the identity $(\ref{eq:661})$ then follows. The expressions for $a^{2m}$ in terms of the expansion for $\binom{m+k-1}{2k-1}$ given in
 $(\ref{eq:665})$ and $(\ref{eq:667})$ follow similarly.
\end{proof}

\begin{defn}
Let $\mathcal{V}$ be the subset of polynomials of $\mathbb{Q}[x]$ of even degree,
spanned by $\{\langle x(x+1) \rangle^{\ul r}\}_{r=0}^{\infty}$, so that for $q(x)\in \mathcal{V}$,
we can express $q(x)$ as a linear combination of the $(x(x+1))^r$ with rational coefficients.

In a similar fashion define $\mathcal{W}$ to be the subset of polynomials of $\mathbb{Q}[x]$ of odd degree over $\mathbb{Q}$,
such that for $w(x)\in \mathcal{W}$, we can express $w(x)$ as a linear combination of the $x^{2r-1}$, with $r\geq 0$. Hence
$w(x)$ can contain an $x^{-1}$ term as demonstrated in the example for $f_7^{(6)}(x)\in \mathcal{W}$ following the proof of Lemma 5.3.

\end{defn}
We now consider some properties
of these polynomials.

\begin{lemma}
Let $q(x)\in \mathcal{V}$ be a polynomial of degree $2n$, so that
\[
q(x)=\sum_{r=0}^n a_r(x(x+1))^r,
\]
with $a_r\in \mathbb{Q},\,\, 0\leq r \leq n$. Then the coefficients of $x^t$, $0\leq t \leq 2n$ in $q(x)$ are given by
\be
\sum_{t=0}^{\hbox{\rm\small{min}}\,(n-r,r-1)}a_{r+t}\binom{r+t}{r-t-1},
\label{eq:60}
\ee
when $t=2r-1$ is odd, and
\be
\sum_{t=0}^{\hbox{{\rm\small{min}}}\,(n-r,r)}a_{r+t}\binom{r+t}{r-t},
\label{eq:61}
\ee
when $t=2r$ is even.
\end{lemma}
\begin{corollary}
The representation for each $p(x)\in \mathcal{V}$ is unique, so that the polynomial sequence $\{( x(x+1) )^{\ul r}\}_{r=0}^{\infty}$ forms a
linearly independent basis for $\mathcal{V}$, and if
\be
p(x)=\sum_{r=0}^n a_r(x(x+1))^r,\qquad q(x)=\sum_{r=0}^N b_r(x(x+1))^r,\qquad p(x)=q(x)+\lambda,
\label{eq:63}
\ee
for some $\lambda\in \mathbb{Q}$, then we have $n=N$, $a_0=b_0+\lambda$, and $a_r=b_r$ for $r\geq 1$.
\end{corollary}
\begin{proof}
The result follows by writing
\[
\sum_{r=0}^n a_r(x(x+1))^r= \sum_{r=0}^n a_r x^r\sum_{j=0}^r x^j\binom{k}{j},
\]
and collecting terms for fixed powers of $x$.

To see the Corollary we assume that
$a_i\neq b_i$ for some $0\leq i\leq n,\, N$, and $a_n\neq 0$, $b_N\neq 0$.
The polynomial $(x(x+1))^r$ contains terms involving $x^r,x^{r+1},\ldots,x^{2r}$, and so by considering the coefficients of the
largest power of $x$, we deduce from $(\ref{eq:61})$ with $r=n$, that $n=N$ and $a_n=b_n$. Considering $(\ref{eq:60})$ with $r=n$
then yields $a_{n-1}=b_{n-1}$. Similarly, continuing with $r=n-1$ in $(\ref{eq:61})$ and then $(\ref{eq:60})$ gives us $a_{n-2}=b_{n-2}$,
and  $a_{n-3}=b_{n-3}$ from which the proof follows by induction on $k$, and we find that
$a_0=b_0+\lambda$, as required.
\end{proof}

\begin{lemma}[Principal polynomial lemma]
Let the coefficients $a_r\in\mathbb{Q}[a]$ and $b_r\in\mathbb{Q}[b]$, be polynomials in $a$ and $b$ respectively, defined by
\be
a_r=\sum_{k=r}^m \left (\frac{Ps_k^{(r)}}{(2k+1)!}\sum_{j=1}^{a-1}(j(j+1))^{m-k}\right ),
\label{eq:m66}
\ee
and
\be
b_r=\sum_{k=r}^m \left (\frac{\mathcal{S}_k^{(r)}}{(2k)!}\sum_{j=1}^{b-1}(-1)^{j+b-1}(j(j+1))^{m-k}\right ),
\label{eq:m665}
\ee
Then for each natural number $m$ and integer $a$ we can write
\be
a^{2m+1}=\sum_{r=0}^m \alpha_r(m(m+1))^r,
\qquad
\alpha_r=\begin{cases} (2m+1)a_r &\mbox{if } 1\leq r \leq n, \\
(2m+1)a_r+1 & \mbox{if } r=0,
\end{cases}
\label{eq:m65}
\ee
and
\be
a^{2m}=\sum_{r=0}^m \beta_r m^{2r-1},
\qquad
\beta_r=\begin{cases} (2m)b_r &\mbox{if } 1\leq r \leq n, \\
(2m)b_r+(-1)^{a-1} & \mbox{if } r=0.
\end{cases}
\label{eq:m651}
\ee
Hence with each integer power $a^k$ we can uniquely associate a polynomial $f_a^{(k)}(x)$, defined by
\[
f_a^{(k)}(x)=\sum_{r=0}^{\left [\frac{k}{2} \right ]}\alpha_r(x(x+1))^r,
\]
when $k$ is odd and by
\[
f_a^{(k)}(x)=\sum_{r=0}^{\left [\frac{k}{2} \right ]}\beta_r x^{2r-1},
\]
when $k$ is even.
\end{lemma}

\begin{proof}
The polynomial representations follow immediately from equations $(\ref{eq:661})$ and~$(\ref{eq:667})$.
\end{proof}
For example, when $k=\in \{5,6\}$, we have the principle polynomials
\[
f_a^{(5)}(x)=\frac{(-1 + a)}{24} (x (x + 1))^2 +
 \frac{ (-1 + a) (3-13a + 10 a^2)}{36} (x (x + 1))
 \]
 \[ +
 \frac{(-1 + a) a (1 + a) (-2 + 3 a^2)}{3}+1,
\]
and $f_a^{(6)}(x)=$
\[
6 \left(\frac{1}{144} (-1)^a \left((-1)^a \left(3 a^2-2\right)+1\right) x^3+\frac{(-1)^a \left((-1)^a \left(360 a^4-750
   a^2+199\right)-191\right) x}{1440}\right.
\]
\[
\left.+\frac{(-1)^a \left((-1)^a \left(4 a^6-18 a^4+24 a^2-5\right)+5\right)}{8 x}+\frac{(-1)^a
   \left((-1)^a+1\right) x^5}{1440}\right)+(-1)^{a-1}.
\]
It follows that $f_7^{(5)}(x)\in \mathcal{V}$ is given by
\[
f_7^{(5)}(x)=5 \left(\frac{1}{20} x^2 (x+1)^2+\frac{557}{30} x (x+1)+3248\right)+1,
\]
and $f_7^{(6)}(x)\in \mathcal{W}$ by
\[
f_7^{(6)}(x)=6 \left(x^3+575 x+\frac{53568}{x}\right)+1.
\]
We note that the two polynomials both exhibit the even and odd power ``shapes'' described at the beginning of this section, and as expected they give $f_7^{(5)}(2)=16807=7^5$, and $f_7^{(6)}(3)=117649=7^6$.

\begin{rmk}
By construction the polynomial $f_a^{(k)}\left (x\right )$ satisfies $f_a^{(k)}\left ([k/2]\right )=a^{k}$, independently of the value of $a$. Hence for $a\in\mathbb{R}$ and $k\in\mathbb{N}$, we have that $f_a^{(k)}(x)$
is an interpolating polynomial with rational coefficients for the number $a^k$, attaining the value $a^k$ when $x=[k/2]$, and we call $f_a^{(k)}(x)$ the \emph{principal polynomial} of $a^{k}$. Whether these polynomials have any practical use is not yet clear although this may depend on whether geometric interpretations of interest can be established.

One possible area for investigating geometric connections is to examine how varying the $x$ term corresponds geometrically to dilating the (possibly truncated) simplices, which the binomial coefficients are the lattice point enumerators of \cite{beck1}, \cite{beck2}.
\end{rmk}

\small{Mark W. Coffey\\
Colorado School of Mines\\
Golden\\
Colorado 80401\\
email: mcoffey@mines.edu}\\
\vspace{5mm}

\noindent
\small{Matthew C. Lettington\\
School of Mathematics\\
Cardiff University\\
Senghennydd Road\\
Cardiff \\
UK, CF24 4AG\\
email: LettingtonMC@cardiff.ac.uk}
\end{document}